\documentclass[12pt]{article}
\usepackage{amsmath,amsfonts,amssymb,amsthm}
\usepackage[usenames,dvipsnames]{pstricks}
\usepackage{epsfig}

\def\noi{\noindent}

\def\pf{\noi{\bf Proof.\ \,}}
\def\eop{{$\square$}}

\def\labtt#1{\label {#1} }
\def\labttr#1{\label {#1}\rm }

\def\refpp#1{(\ref{#1})}

\def\a{\alpha}
\def\b{\beta}
\def\g{\gamma}
\def\d{\delta}
\def\l{\lambda}

\def\s{\sigma}

\def\th{\theta}

\def\L{\Lambda}

\def\CC{{\mathbb C}}
\def\FF{{\mathbb F}}
\def\MM{{\mathbb M}}
\def\QQ{{\mathbb Q}}
\def\RR{{\mathbb R}}

\def\TT{{\mathbb T}}
\def\ZZ{{\mathbb Z}}

\def\la{\langle}
\def\ra{\rangle}
\def\<{\langle}
\def\>{\rangle}

\def\bs{\it}            % Bob Style
\def\Aut{{\bs Aut}}
\def\dim{{\bs dim}}

\def\exp{{\bs exp}}

\def\Sym{{\bs Sym}}
\def\l{{\lambda}}

\def\half{{1 \over 2}}

\def\third{{1\over 3}}

\def\dual#1{#1^*}        %% the dual of the lattice #1

\def\weh{Weyl({E_8})}

\def\dg#1{{\cal D}({#1})}
\def\vnat{V^\natural}

\def\cvcch{cvcc$\half$ }
\def\psia{\psi a}
\newcommand{\sfr}[2]{\leavevmode\kern-.1em
  \raise.5ex\hbox{\the\scriptfont0 #1}\kern-.1em
  /\kern-.15em\lower.25ex\hbox{\the\scriptfont0 #2}}
\begin{document}

\newtheorem{thm}{Theorem}[section]
\newtheorem{prop}[thm]{Proposition}
\newtheorem{lem}[thm]{Lemma}
\newtheorem{rem}[thm]{Remark}
\newtheorem{coro}[thm]{Corollary}
\newtheorem{conj}[thm]{Conjecture}
\newtheorem{de}[thm]{Definition}
\newtheorem{hyp}[thm]{Hypothesis}

\newtheorem{nota}[thm]{Notation}
\newtheorem{ex}[thm]{Example}
\newtheorem{proc}[thm]{Procedure}

\begin{center}\end{center}

\centerline{ \today  }
\begin{center}
{\Large  A moonshine path from $E_8$ to the monster}

\bigskip 

(previous title of preprint: A $3C$-path for Glauberman-Norton theory)

\begin{center}

\vspace{10mm}
Robert L.~Griess Jr.
\\[0pt]
Department of Mathematics\\[0pt] University of Michigan\\[0pt]
Ann Arbor, MI 48109  USA  \\[0pt]
{\tt rlg@umich.edu}\\[0pt]
\vskip 1cm

Ching Hung Lam
\\[0pt]
Institute of Mathematics \\[0pt]
Academia Sinica\\[0pt]
Taipei 115, Taiwan\\[0pt]
{\tt chlam@math.sinica.edu.tw}\\[0pt]
\vskip 1cm
\end{center}

\begin{abstract}
One would like an explanation of the provocative McKay and Glauberman-Norton observations  connecting the extended $E_8$-diagram with pairs of $2A$ involutions in the Monster sporadic simple group.  We propose a down-to-earth model for the $3C$-case which exhibits a logic to these connections.
\end{abstract}

\end{center}

\tableofcontents

\section{Introduction}

In 1979, John McKay \cite{mckay} noticed a remarkable correspondence between
$\tilde{E_8}$, the extended
$E_8$-diagram, and pairs of  $2A$-involutions in $\mathbb{M}$, the Monster (the
largest sporadic finite simple group).

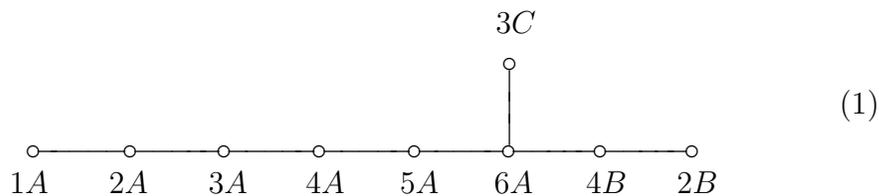
\begin{equation}\label{McKay diagram}
\begin{array}{l}
  \hspace{184pt} 3C\\
  \hspace{186.2pt}\circ \vspace{-6.2pt}\\
   \hspace{187.5pt}| \vspace{-6.1pt}\\
 \hspace{187.5pt}| \vspace{-6.1pt}\\
 \hspace{187.5pt}| \vspace{-6.2pt}\\
  \hspace{6pt} \circ\hspace{-5pt}-\hspace{-7pt}-\hspace{-5pt}-
  \hspace{-5pt}-\hspace{-5pt}-\hspace{-5pt}\circ\hspace{-5pt}-
  \hspace{-5pt}-\hspace{-5pt}-\hspace{-6pt}-\hspace{-7pt}-\hspace{-5pt}
  \circ \hspace{-5pt}-\hspace{-5.5pt}-\hspace{-5pt}-\hspace{-5pt}-
  \hspace{-7pt}-\hspace{-5pt}\circ\hspace{-5pt}-\hspace{-5.5pt}-
  \hspace{-5pt}-\hspace{-5pt}-\hspace{-7pt}-\hspace{-5pt}\circ
  \hspace{-5pt}-\hspace{-6pt}-\hspace{-5pt}-\hspace{-5pt}-
  \hspace{-7pt}-\hspace{-5pt}\circ\hspace{-5pt}-\hspace{-5pt}-
  \hspace{-6pt}-\hspace{-6pt}-\hspace{-7pt}-\hspace{-5pt}\circ
  \hspace{-5pt}-\hspace{-5pt}-
  \hspace{-6pt}-\hspace{-6pt}-\hspace{-7pt}-\hspace{-5pt}\circ
  \vspace{-6.2pt}\\
  \vspace{-10pt} \\
  1A\hspace{23pt} 2A\hspace{23  pt} 3A\hspace{22pt} 4A\hspace{21pt}
  5A\hspace{21pt} 6A\hspace{20pt} 4B\hspace{19pt} 2B\\
\end{array}
\end{equation}
There are 9 conjugacy classes of such
pairs $(x,y)$, and the
orders of the 8 products $|xy|$, for $x\ne y$,  are the coefficients of
the highest root in the $E_8$-root system.  Thus, the 9 nodes are labeled with 9
conjugacy classes of $\mathbb M$.
There is no obvious reason why there should be such a correspondence involving
high-level theories from different parts of the mathematical universe.

In 2001, George Glauberman and Simon Norton \cite{gn}
enriched this theory by
adding details about the centralizers in the Monster of such pairs of
involutions and relations involving the associated modular forms.
Let $(x,y)$ be such a pair and let $n(x,y)$ be its associated node.  Let
$n'(x,y)$ be  the subgraph of $\tilde E_8$ which is supported at the set of
nodes
complementary to $\{n(x,y)\}$.
If $(x,y)$ is a  pair of  $2A$ involutions and $z$ is a $2B$ involution which
commutes with $\la x, y \ra$, Glauberman and Norton
 give a lot of detail about $C(x, y, z)$.  In particular, they explained how
$C(x, y, z)$ has a ``new'' relation to the extended $E_8$-diagram, namely
that
$C(x, y, z)/O_2(C(x, y, z))$   looks roughly like ``half'' of
the Weyl group corresponding to the subdiagram  $n'(x,y)$.

The important and provocative McKay-Glauberman-Norton observations seemed
like looking across a great foggy space,
from one high mountain top to another.
We want to realize their connections in a manner which is more down-to-earth,
like walking along a path, making natural
steps with familiar mathematical objects.
These objects are lattices, vertex operator algebras, Lie algebras,
Lie groups and finite groups.

In this paper, we propose a specific path for the $3C$-case (i.e.,
$n'(x,y)$ is an $A_8$-diagram). The $3C$-case seems to be especially
rich.  Several Niemeier lattices are involved. They include $E_8^3$
and the Leech lattice $\Lambda$. Triality for $D_4$ plays a role. An
explanation for occurrence of  just ``half'' the Weyl group (of type
$A_8$) arises naturally.  We hope to develop similar paths for other nodes.

\newpage
\subsection
{Compact Summary of Strategy}

This subsection contains a brief outline of   how one may start with a node of the extended diagram $\widetilde {E_8}$ and move to a pair of $2A$-involutions in the monster, $\MM$.

For simplicity, we describe two paths, one beginning with a node of $\widetilde {E_8}$ and the second one beginning with a pair of $2A$-involutions in $\MM$.  Each path ends with a subVOA  generated by a pair of conformal vectors,
for which theories on dihedral subVOAs
give isomorphisms and enable us to splice the paths.
Our Glauberman-Norton path consists of the path from $\widetilde {E_8}$ followed by the reverse of the above path $\MM$.

\bigskip

\noindent {\bf Path starting in $\widetilde {E_8}$:}

\noindent node
$\rightarrow$

\smallskip

\noindent sublattice $K$ of finite index in $E_8$
$\rightarrow$

\smallskip

\noindent element $r\in E_8(\CC )$ of order $|E_8 : K|$, defined by exponentiation
$\rightarrow$

\smallskip

 \noindent \cvcch $e, f$ in $V_{EE_8} \le V_{E_8\oplus E_8}$
$\rightarrow$

\smallskip

\noindent conjugacy of $r$ to an
element $h$ in torus normalizer $N (\TT )$ \\
so $h$ acts on the root lattice without eigenvalue 1
$\rightarrow$

\smallskip

\noindent a pair of $EE_8$ lattices $M, M' < E_8^3$  and \cvcch
$e_M, e_{M'}$ such that\\ $subVOA\la e_M, e_{M'}\ra \cong subVOA\la
e,f \ra$ $\rightarrow$

\smallskip

\noindent Niemeier lattice $N$ with automorphism $h''$ so that
$N^+(h'')$ and  $N_+(h'')$ are related to $K$; find overlattice of
$N^+(h'') \perp N_+(h''))$ isometric to Leech lattice.

\bigskip

\noindent {\bf Path starting in $\MM$:}

\noindent distinct $2A$-involutions $x, y \in \MM$
$\rightarrow$

\smallskip

\noindent $x, y$ correspond to unique \cvcch $e', f'$ (Miyamoto bijection) in $\vnat$; we may replace $x, y$ by conjugates to take $e, f$ in $V_{\L}^+$  

\bigskip

\noindent {\bf At  the endpoints of these two paths}

\noindent Existing results on dihedral subalgebras of VOAs prove that
$subVOA\la e', f'\ra \cong subVOA\la e, f \ra$ if and only if
$n(x,y)$ is the node in the $\widetilde {E_8}$ procedure \cite{LYY2,Sa}.
\bigskip

\noindent {\bf Observation:  }

\noindent We use triality for $D_4$ to find a Leech lattice
as exceptional overlattice of
$N^+(h') \perp N_+(h'))$, resulting in visible
 loss of half the Weyl group, going from $n(x,y)$ to $x, y$ (and then on to $C(x,y,z)$).

\subsection{Details on steps}

Our Glauberman-Norton path starting in $\widetilde {E_8}$ involves several steps, which we preview here.

\noindent \textbf{Step I.} We show that the subdiagram $n'(x,y)$ defines an
automorphism $r=r(x,y)$ of exponential type in $Aut(V_{E_8})$. Then we construct
a pair of conformal vectors of central charge $1/2$ (abbreviated as \cvcch) $e$ and $f$ in $V_{EE_8}$.

Let $e\in V$  be a \cvcch, i.e., the subVOA $Vir(e)$ generated by
$e$ is isomorphic to $L(\sfr{1}2,0)$. It is well known that
$L(\sfr{1}2,0)$ is rational, $C_2$-cofinite and has three
irreducible $L(\sfr{1}2,0)$, $L(\sfr{1}2,\sfr{1}2)$ and
$L(\sfr{1}2,\sfr{1}{16})$ (cf. \cite{dmz}).

Let $V_e(h)$ be the sum of all irreducible $Vir(e)$-submodules of
$V$ isomorphic to $L(\sfr{1}2,h)$ for $h=0,1/2,1/16$. Then one has
an isotypical decomposition:
$$
  V=V_e(0)\oplus V_e(\sfr{1}2)\oplus V_e(\sfr{1}{16}).
$$
Define a linear automorphism $\tau_e$ on $V$ by
$$
  \tau_e=
  \begin{cases}
    \ \  1 & \text{ on }\  V_e(0)\oplus V_e(\sfr{1}2),\\
    -1& \text{ on }\  V_e(\sfr{1}{16}).
\end{cases}
$$
Miyamoto \cite{M1} showed that $\tau_e$ defines an automorphism of
the VOA $V$. This automorphism is often called \textit{the Miyamoto
involution} associated to $e$. It is also known that
$2A$-involutions of $\mathbb{M}$ are in one-to-one correspondence
with conformal vectors of central charge $1/2$ in $V^\natural$
through the construction of Miyamoto involutions \cite{Co,M1}. Thus,
given  a pair of 2A-involutions $x,y$, one can associate a pair of
conformal vectors $e',f'\in V^\natural$ of central charge $1/2$ so
that $x$ is the Miyamoto  involution for $e'$ and $y$ is the
Miyamoto involution for $f'$. By using the above correspondence, one
can show that the dihedral group $\la x,y\ra$ is uniquely determined
by the subVOA generated by $e'$ and $f'$ \cite{Co,Sa,LYY2}.

The diagram $n'(x,y)$
defines an automorphism $r(x,y)$ of $V_{E_8}$ induced by a
character $\lambda$ of $E_8$ with $K= Ker(\lambda)$, i.e.,
\[
r(x,y)(u\otimes e^\a) = \xi^{\lambda(\a+K)} u\otimes e^{\a} \quad \text{for }
u\in M(1), \a \in E_8,
\]
where $\xi$ is a primitive $n$-th
root of unity, $n=|E_8/K|$, $K$ is the root lattice associated to the diagram
$n'(x,y)$ and $M(1)$ is an irreducible $\hat{\mathfrak{h}}$-module.
Here $\mathfrak{h}=\CC\otimes_\ZZ E_8$ and $\hat{\mathfrak{h}}$ is the affine
Lie algebra of $\mathfrak{h}$ (see Section 2 for details).

Let $L=E_8\oplus E_8$ and $M=\{(x,x)|\ x\in E_8\}$. Then $M\cong
EE_8$. Let $e=e_M$ be the \cvcch  defined in Notation \ref{def:phisubx} and
$f=(r(x,y)\otimes 1)(e)$. Then both $e$ and $f$ are \cvcch.

The key observation for this step is the following proposition.
\begin{prop}[cf. \cite{LYY2,Sa}]\label{prop:1.1}
The subVOA $\la e, f\ra$ generated by $e$ and $f$ in $V_{EE_8}$ is isomorphic to
the subVOA in $V^\natural$ generated by the conformal vectors associated to the
$2A$ involutions $x$ and $y$. 
Moreover, the centralizer
of the dihedral group
$\la \tau_e, \tau_f\ra$ in $Aut(V_{EE_8})$ is isomorphic to $2^8{\cdot }Sym_9$, where $\Sym_9$ is the Weyl
group of $A_8$.
\end{prop}

\medskip

\noindent \textbf{Step II.} We explain that $r(x,y)$ is conjugate in
$Aut(V_{E_8})$ to an automorphism $\hat{h}(x,y)$ in a torus normalizer in
$Aut(V_{E_8})\cong E_8(\CC)$
such that $\hat{h}(x,y)$ induces a fixed point free
isometry $h$ on $E_8$ by the natural action of the torus normalizer on the root
lattice. We then derive a pair of
$EE_8$-sublattices $M$ and $M'$ in $E_8^3$ as follows.

Set $\rho:=r(x,y)\otimes 1\otimes 1$ and $\eta:=\eta(x,y):=\hat{h}(x,y)\otimes
1\otimes 1$.
We
identify $V_{E_8^3}$ with $V_{E_8}^{\otimes 3}$,
so that $\rho$ and $\eta$ may be considered
automorphisms of $V_{E_8^3}$.
We also take the two $EE_8$-sublattices of $E_8^3$:
\[
 M=\{(a,a,0)| a\in E_8\}\quad \text{ and }\quad M'=\{(ha, a,0)| a\in E_8\}.
\]

The following are the main results of this step.

\begin{thm}
 $\rho$ is conjugate in $Aut(V_{E_8^3})$ to $\eta$ and $\eta$ is in a torus
normalizer.
\end{thm}

\begin{thm}
Let $e_M$ and $e_{M'}$ be \cvcch supported at $M$ and $M'$,
respectively (cf. Notation \ref{def:phisubx}). Then, the subVOA $\la
e_M, e_{M'}\ra $ generated by $e_M$ and $e_{M'}$  is isomorphic to
$\la e, f\ra$.
\end{thm}
Therefore, we may transfer the study of the dihedral group $\la x,y\ra<
\mathbb{M}$ to the study of  \cvcch $e_M$ and $e_{M'}$ in
$V_\L^+ \subset V^\natural$.

We trade $\rho$ for $\eta$ since $\eta$ looks like a ``permutation of roots''
and gives a
map on a lattice, so can be interpreted as a map on the VOA
$V_\Lambda^+$ associated with the Leech lattice $\L$, whereas $\rho$ is
``exponential'', so cannot have a direct
interpretation as an exponential on $V_\L^+$ (since this
VOA has a finite automorphism group).

\medskip

\noindent \textbf{Step III.} In this step, we  shall take the pair $x, y$
to a pair of Miyamoto
involutions associated to conformal vectors $e_M, e_{M'}$ of central charge
$\half$ which lie in $V_\L^+$.

We first determine the isometry type of $Q:=M+M'$ and show that $Q$ can be
embedded into the Leech lattice $\Lambda$.
The main theorem is as follows.

\begin{thm}
The Leech lattice $\Lambda$ contains a sublattice isometric to $Q\cong
A_2\otimes E_8$ and hence  $U=\la e , f \ra$, the subVOA generated by $e$
and $f$, can be embedded into $V_\Lambda^+$.
Moreover, the annihilator
$R:=ann_\L(Q)$ of $Q$ in $\Lambda$ is isometric to $\sqrt{3}E_8$.
\end{thm}
As a consequence, the subVOA generated by $e_M$ and $e_{M'}$ can be
embedded into $V_\L^+$.
Recall that the moonshine VOA $V^\natural$ is
constructed by \cite{FLM} as a $\ZZ_2$-orbifold of the Leech lattice
VOA $V_\Lambda$, that means,
\begin{equation}\label{Vn}
V^\natural=
(V_\Lambda)^+ \oplus (V_\Lambda^T)^+,
\end{equation}
where $V_\Lambda^T$ is the unique irreducible $\theta$-twisted
module for $V_\Lambda$ and $(V_\Lambda^T)^+$ is the fixed point
subspace of $\theta$ in $V_\Lambda^T$. Thus $U=\la e , f \ra$ can
also be embedded into the Moonshine VOA $V^\natural$. We shall note
that $\eta$ leaves the subVOA $V_Q$ invariant. Thus, it induces an
automorphism $\eta_Q$ on $V_Q$ by restriction.
We also show that 
$\eta_Q$ can be extended to an automorphism $\eta_{\Lambda}$ in
$Aut(V_\Lambda)$. Thus, $\eta$ has a life on $V_{\L}$ and
$V_\Lambda^+$. Since $V_\Lambda^+\cong (V^\natural)^z$ for a $2B$
involution  $z \in \MM$ and $Aut(V_\Lambda^+)\cong C_\MM(z)/\la z\ra
\cong 2^{24}\cdot Co_1$ , we can study the centralizer of $\la x, y,
z\ra$ in $\mathbb{M}$ by using the configuration of $M$, $M'$ and
their sum $Q$ in $\L$. This leads us to study the overlattices of
$Q\perp R$ and the corresponding gluing maps.
It turns out that the
stabilizer of a gluing map is exactly the normalizer of $\eta$ in
the isometry group of the overlattice \refpp{stab1}.

\medskip

\noindent \textbf{Step IV.}
Our analysis at the stage where we enlarge $Q\perp R$ to $\L$ leads to an
analysis of gluing maps.  There exists one whose stabilizer is a subgroup
$Sym_3 \times 2{\cdot}Alt_9$. Our
proof makes use of triality for groups of type $D_4$.   Since the half-spin
representations play a role, it is clear that we lose the `outer' part of our
subgroup of type $Sym_9$.

In this step, we first start with a gluing map $\a: \dg{Q} \to
\dg{R}$ such that the associated overlattice $L_{\a}$ is isometric
to $E_8^3$. We also construct a subgroup $K_0$ of $Spin^+(8,3)$ so
that $K_0$ is a covering group of $Sym_9$, $K:=K_0'\cong 2{\cdot}
Alt_9$ and $K_0/K \cong 2$.  The main idea is to choose such a $K_0$ so that the action of $K$ comes from a subgroup of $O(Q)\times O(R)$, but not so for $K_0$.
We then
twist $\a$ by an
element $u\in K_0\setminus K$ to get a new gluing map
$\beta=:u\a u^{-1}$. The result is:
\begin{thm}
The associated overlattice $L_\b$ is even unimodular and rootless, so is isometric to
the Leech lattice.  Its stabilizer is a subgroup $Sym_3\times
2{\cdot} Alt_9$ of the group $O(Q\perp R)\cong Sym_3\times
O(E_8)\times O(E_8)$.
\end{thm}

Thus we can, in a sense, witness loss of half the Weyl group of type
$A_8$ for the node $n(x,y)$.  This is an explanation for one of the Glauberman-Norton observations.
\medskip

Note that $u$ gives a map from $\dg{Q\perp R}$ to itself. Hence, it
induces a permutation on the set of all irreducible modules for
$V_{Q\perp R}$ since the irreducible modules for $V_{Q\perp R}$ are
parametrized by $\dg{Q\perp R}$ \cite{Dong}. Therefore, the
construction of $L_\b(\cong \L)$ from $L_\a$ can also be interpreted
as an orbifold construction of $V_{L_\b}$ from $V_{L_\a}$ using a
subgroup $A\cong 3^8$ of $Aut(V_{L_\a})$ such that the fixed point
subVOA $(V_{L_\a})^A$ is isomorphic to $V_{Q\perp R}$ (cf.
\cite{dlm96,dlm,Li}).

This ends the preview of our $3C$-path construction.  It begins with
one set of data (the extended $E_8$ diagram) and ends with $V_{\L}$.
In the latter VOA, we find concrete realizations of the second set
of data (dihedral groups generated by pairs of $2A$-involutions),
namely pairs of conformal vectors of central charge $\half$ which
represent all 9 types of these dihedral groups. The monster group
does not act as automorphisms of this VOA, but rather does so on an
orbifold of it, called $\vnat$.  Both $V_{\L}$ and $\vnat$ contain a
subVOA $V_{\L}^+$, where suitable pairs of conformal vectors may be
found (so we felt no need to add details about $\vnat$ in this
article).
In \cite{LSY}, all \cvcch in the VOA $V_{\L}^+$ were
classified. There are two types of \cvcch. The first type ($AA_1$-type) is
associated to  a norm $4$ vector $\a$  in $\L$ and denoted by $\omega^\pm(\a)$
(cf. Notation \ref{cvcchAA1EE8}).
The corresponding Miyamoto involution is defined by
\[
\tau_{\omega^\pm(\a)}(u\otimes e^\b) =(-1)^{\la \a, \b\ra} u\otimes e^\b \quad \text{ for }
u\in M(1), \b\in \L.
\]
The second type ($EE_8$-type) is associated to an $EE_8$-sublattice $M$ of $\L$.
The corresponding Miyamoto involution induces an isometry of $\L$,
which acts as $-1$ on $M$ and $1$ on $ann_\L(M)$ (see Notation
\ref{def:phisubx} and Appendix \ref{Appendix:B}). Our recent
classification \cite{glee8era,glee8} of configurations of
$EE_8$-lattices is used to analyze relevant pairs of conformal
vectors.

Building materials for our path come from several highly developed
mathematical theories  (Lie theory, lattices, vertex operator
algebras, finite groups).  More aspects  of these theories could
play roles in the future.  We hope for a wide moonshine road, making
the study of moonshine more concrete and enabling the transporting
of ideas.  In particular, this ought to illuminate connections
between the extended  $E_8$-diagram and the monster.

\medskip

The first author thanks National Cheng Kung University for financial
support during a visit to Tainan, Taiwan, and the U. S. National
Science Foundation for support from grant NSF (DMS-0600854).  The
second author thanks National Science  Council of Taiwan for support
from grant NSC 97-2115-M006-015-MY3.

\subsection{Notation and Terminology}

In this article,  all group actions are assumed to be on the left.
Our notation for the lattice vertex operator algebra
\begin{equation}\label{VL}
V_L = M(1) \otimes \CC[L]
\end{equation}
associated with a positive definite even lattice $L$ is also
standard \cite{FLM}. In particular, ${\mathfrak h}=\CC\otimes_{\ZZ}
L$  is an abelian Lie algebra and we extend the bilinear form to
${\mathfrak h}$ by $\CC$-linearity. Also, $\hat {\mathfrak
h}={\mathfrak h}\otimes \CC[t,t^{-1}]\oplus \CC k$ is the
corresponding affine algebra and $\CC k$ is the 1-dimensional center
of $\hat{\mathfrak{h}}$. The subspace $M(1)=\CC[\a(n)|\a\in
{\mathfrak h}, n<0],$ where $\a(n)=\a\otimes t^n,$ is the unique
irreducible  $\hat{\mathfrak h}$-module such that $\alpha(n)\cdot
1=0$ for all $\alpha\in {\mathfrak h}$ and $n$ positive, and $k=1.$
Also,  $\CC[L]=\{e^{\beta}\mid \beta\in L\}$ is the twisted group
algebra of the additive group $L$ such that $e^\b e^\a=(-1)^{\la \a,
\b\ra} e^\a e^\b$ for any $\a, \b\in L$. The vacuum vector
$\mathbf{1}$  of $V_L$ is $1\otimes e^0$ and the Virasoro element
$\omega$ is $\frac{1}{2}\sum_{i=1}^d\beta_i(-1)^2\cdot \mathbf{1}$
where $\{\beta_1,..., \beta_d\}$ is an orthonormal basis of
${\mathfrak h}.$ For the explicit definition of the corresponding
vertex operators, we shall refer to \cite{FLM} for details.

\begin{center}
{\bf Notation and Terminology}

\medskip
\begin{tabular}{|c|c|c|}
   \hline
\bf{Notation}& \bf{Explanation} & \bf{Examples in text}  \cr
   \hline\hline
$2A, 2B, 3A, \dots$ & conjugacy classes of the Monster, &  Equation \eqref{eq:1.1}\cr
                  & the first number denotes the order & \cr
                  & of the elements and the second letter & \cr
                  & is arranged in descending order of & \cr
                  & the size of the centralizers &\cr \hline
$A_1,  \cdots , E_8$ &root lattice for root system &  Sec. 2\cr
& $\Phi_{A_1}, \dots ,
\Phi_{E_8}$& \cr
   \hline
$AA_1, \cdots ,$ & lattice isometric to $\sqrt 2$ times &\cr
$ EE_8$ &the lattice  $A_1, \cdots,  E_8$ & Sec. 2\cr
\hline
$AAA_1, \cdots ,$ & lattice isometric to $\sqrt 3$ times & \cr
$ EEE_8$ & the lattice  $A_1, \cdots,  E_8$ & Remark \ref{qinleech}\cr
\hline
\end{tabular}
\end{center}
\newpage

\begin{center}
 \begin{tabular}{|c|c|c|}
  \hline
\bf{Notation}& \bf{Explanation} & \bf{Examples in text}  \cr
   \hline\hline
$E_{i,j}$ & a square matrix whose $(i,j)$-th entry  & Sec. \ref{sl}, \cr
 & is $1$ and all other entries are $0$& Equation \eqref{eta} \cr \hline
$e_M$& a principal conformal vector of $V_M$, i.e,  & Notation \ref{def:phisubx}\cr
&  $ e_M=\frac{1}{16}\omega_M + \frac{1}{32}
\sum_{\a\in M(4)} e^\a $, &\cr
& where $M\cong EE_8$& \cr
\hline
    \vspace{-0.35cm}
\ &&  \cr
$\eta$ or $\eta_{E_8^3}$ & the automorphism $\tilde{h}\otimes 1\otimes 1$ of
& Notation \ref{rhoandeta}\cr
&$V_{E_8^3}\cong
V_{E_8}\otimes V_{E_8}\otimes V_{E_8}$ & \cr
\hline
$\eta_Q$ & the restriction of $\eta$ to $V_Q\subset V_{E_8^3}$& Step III of \cr
&& Introduction\cr \hline
$h$ & a fixed point free automorphism&\cr
& of $E_8$ of order $3$ & Notation \ref{hinE8}\cr \hline
\vspace{-0.4cm}
&&\cr
$\tilde{h}$ & a lift of $h$ in $V_{E_8}$, i.e, $\tilde{c}( M(1) )\subset
M(1)$,& Equation \eqref{tildeh}\cr
&  $\tilde{h}(e^x)=\epsilon_x e^{h(x)}$,
$\epsilon_x=\pm 1$ &\cr
\hline
 $h_{A_n}$ & a Coxeter element in $Weyl(A_n)$ & Sec. \ref{sl}, \cr
 && Equation \eqref{hAn} \cr\hline
\vspace{-0.4cm}
&&\cr
$\tilde{h}_{A_n}$ & a lift of $h_{A_n}$ in $Aut(sl_{n+1}(\CC))$ (See & Equation \eqref{eta}\cr
& Equation \eqref{eta} for the precise definition) &\cr \hline
$K$ or $K_{nX}$& the lattice associated with &Equation \eqref{K}\cr
&the Dynkin subdiagram of $\hat{E_8}$&\cr
& with the $nX$-node removed&\cr
\hline
$L^+(\th), L_+(\th )$&the fixed point sublattice & Proposition \ref{stab0}\cr
&of theta, its annihilator, resp. & \cr
\hline
$L(k)$ & the set of all norm $k$ vectors in $L$,   &  Notation \ref{def:phisubx}\cr
&i.e., $ L(k)=\{ a\in L\mid \la a, a\ra =k\}$ &\cr \hline

$\MM$ & the Monster simple group & Compact Summary, \cr
&& Appendix \ref{App:D}     \cr      \hline
$M(\phi )$ & overlattice defined by gluing map $\phi$ & Notation \ref{mphi}
\cr \hline
  Niemeier  & a rank 24 even unimodular lattice
  &Introduction, \cr
  lattice&&Appendix \ref{QRinN}\cr \hline
 $N(X)$ &
Niemeier lattice whose root system& Appendix \ref{QRinN}\cr &has
type $X$& \cr \hline
  $O(X)$ & the isometry group of  & Remark
\ref{qinleech}, \cr &the quadratic space $X$& Lemma \ref{stab0}\cr
\hline
 $O(X, Y, \dots )$&$O(X)\cap O(Y)\dots $ & %\refpp{OXY}
 \cr \hline
\end{tabular}
\end{center}
\newpage

\vspace{-1cm}
\begin{center}
 \begin{tabular}{|c|c|c|}
  \hline
\bf{Notation}& \bf{Explanation} & \bf{Examples in text}  \cr
   \hline\hline
$\varphi_x$ & an automorphism of $V_{EE_8}$ defined by&
Equation \eqref{phi_x}\cr & $\varphi_x(u\otimes e^\a) =(-1)^{\la x,\a\ra}
u\otimes e^\a$ &\cr & for $u\in M(1)$ and $\a \in EE_8$& \cr \hline
$Q$ & $A_2\otimes_{\ZZ}E_8$,  lattice isometric to  &  Notation \ref{QandR}\cr
  & the tensor product
of $A_2$ and $E_8$ &\cr \hline
$R$ &   $EEE_8$, lattice isometric to  &  Notation \ref{QandR}, \cr
  &$\sqrt 3$ times the
root lattice $E_8$ & Remark \ref{qinleech}\cr
\hline
  $r$ or $r({nX})$& an automorphism of $V_{E_8}$ induced  & Notation \ref{r}\cr
&by a character  of
 $E_8/K_{nX}$&\cr
\hline
$r_M$ or $r_M(nX)$ & an automorphism of $V_M$, $M\cong EE_8$ & Equation \eqref{rM} \cr
& induced by a character  of  $E_8/K_{nX}$&\cr
\hline
$r_{A_n}$ & an automorphism of $sl_{n+1}(\CC)$ defined by & Equation \eqref{g}, \cr
  &  $r_{A_n}(E_{i,j}) = \omega^{i-j} (E_{i,j})$, $\omega=e^{2\pi i /(n+1)}$
& Equation \eqref{rA2}\cr \hline
$\rho$ & the automorphism $r\otimes 1\otimes 1$ of & Notation \ref{rhoandeta}\cr
&$V_{{E_8}^3}\cong
V_{E_8}\otimes V_{E_8}\otimes V_{E_8}$ & \cr

\hline

$s_{A_n}$ & an automorphism of $sl_{n+1}(\CC)$ & Def. \ref{sAnth}\cr
&such that $r_{A_n}= s_{A_n} \tilde{h}_{A_n} s_{A_n}^{-1}$ &\cr
\hline $s$ & $s=s_{A_2}\otimes s_{A_2}\otimes s_{A_2}\otimes
s_{A_2}$ &Equation \eqref{rands}\cr & is an automorphism of
$V_{E_8}$& \cr\hline
%%%%%%%%%
$\sigma$ & $\sigma=s\otimes s\otimes s$,  & Notation \ref{sigma}\cr
 & an automorphism of $V_{E_8^3}$& \cr \hline
 %%%%%%%%%%
$\tau_e$ & the Miyamoto involution associated  & Prop
\ref{prop:1.1}\cr & to a conformal vector $e$, i.e.,&\cr & $\tau_e$
acts as $-1$ on $W_{1/16}$ & \cr & and $1$ on $W_{0}\oplus W_{1/2}$,
&\cr & where $W_h$ is the sum of all & \cr & irreducible
$Vir(e)$-submodules &\cr &  isomorphic to $L(1/2,h)$,&\cr & $h=0,
1/2, 1/16$& \cr \hline
%%%%%%%%%%%%%
$\theta$ or $\theta_L$ & an involution of $V_L$ defined by& Equation
\eqref{theta}\cr & $\theta(x_1(-n_1)\dots x_k(-n_k)\otimes e^x)=$
&\cr &$ (-1)^{k+\la x,x\ra/2} x_1(-n_1)\dots x_k(-n_k)\otimes e^{-x}
$&\cr \hline
 %%%%%%%%%%%%
   $V_L$ & the lattice VOA associated with & Equation \eqref{VL}\cr
   & an even lattice $L$& \cr
\hline
%%%%%%%%%
 $Weyl(A_n),\cdots,$ & the Weyl group of the  &
  Equation \eqref{hAn},\cr
  $ Weyl(E_8)$& corresponding root system& Equation
 \eqref{rands}\cr \hline
 \end{tabular}
\end{center}

\newpage

\section{McKay's $E_8$ diagram and Leech lattice}\label{sec:2}

We now set up notation for the $3C$ case and establish our path.

Consider the McKay diagram.
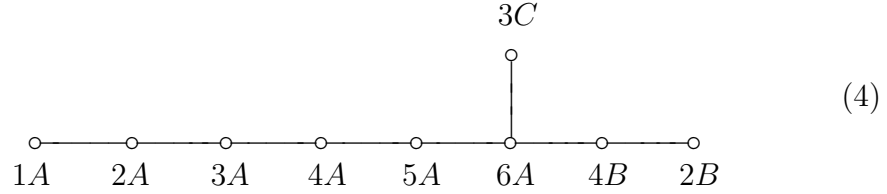
\begin{equation}\label{eq:1.1}
\begin{array}{l}
  \hspace{184pt} 3C\\
  \hspace{186.2pt}\circ \vspace{-6.2pt}\\
   \hspace{187.5pt}| \vspace{-6.1pt}\\
 \hspace{187.5pt}| \vspace{-6.1pt}\\
 \hspace{187.5pt}| \vspace{-6.2pt}\\
  \hspace{6pt} \circ\hspace{-5pt}-\hspace{-7pt}-\hspace{-5pt}-
  \hspace{-5pt}-\hspace{-5pt}-\hspace{-5pt}\circ\hspace{-5pt}-
  \hspace{-5pt}-\hspace{-5pt}-\hspace{-6pt}-\hspace{-7pt}-\hspace{-5pt}
  \circ \hspace{-5pt}-\hspace{-5.5pt}-\hspace{-5pt}-\hspace{-5pt}-
  \hspace{-7pt}-\hspace{-5pt}\circ\hspace{-5pt}-\hspace{-5.5pt}-
  \hspace{-5pt}-\hspace{-5pt}-\hspace{-7pt}-\hspace{-5pt}\circ
  \hspace{-5pt}-\hspace{-6pt}-\hspace{-5pt}-\hspace{-5pt}-
  \hspace{-7pt}-\hspace{-5pt}\circ\hspace{-5pt}-\hspace{-5pt}-
  \hspace{-6pt}-\hspace{-6pt}-\hspace{-7pt}-\hspace{-5pt}\circ
  \hspace{-5pt}-\hspace{-5pt}-
  \hspace{-6pt}-\hspace{-6pt}-\hspace{-7pt}-\hspace{-5pt}\circ
  \vspace{-6.2pt}\\
  \vspace{-10pt} \\
  1A\hspace{23pt} 2A\hspace{23  pt} 3A\hspace{22pt} 4A\hspace{21pt}
  5A\hspace{21pt} 6A\hspace{20pt} 4B\hspace{19pt} 2B\\
\end{array}
\end{equation}

By removing the node labeled $3C$, the remaining subdiagram is a Dynkin
diagram of type $A_8$.

Let $M\cong EE_8$ and $\hat{M}=\{ \pm e^\a|\ \a \in M\}$ a central
extension of $M$ by ${\pm1}$ such that $e^\a e^\b =\pm e^{\a+\b}$
and $e^\a e^\b=(-1)^{\langle \a, \b\rangle} e^\b e^\a$. Since
$M\cong EE_8$ is doubly even, we may arrange that $\{e^\a \mid \a
\in M\}$ forms a subgroup of $\hat{M}$ \cite{FLM}.

Let $K$ be a sublattice of $M$ which is isometric to $AA_8$.
Then $|M/K|=3$, say
\begin{equation}\label{K}
 M=K\cup (\b+K)\cup (-\b +K), \text{ for some }\b \in M\setminus K.
\end{equation}
Then the lattice VOA $V_M$ decomposes as
\[
 V_M=V_K\oplus V_{\b+K} \oplus V_{-\b+K}
\]
and we can define an automorphism $r_M$ of $V_M$ by
\[
 r_M:=
\begin{cases}
 1 & \text{ on } V_K,\\
 \xi & \text{ on } V_{\b+K},\\
 \xi^2 & \text{ on } V_{-\b+K},
\end{cases}
\]
where $\xi:=e^{2\pi i/3}$. Note that
\begin{equation}\label{rM}
 r_M=\exp(2\pi i \gamma_0)
\end{equation}
for some $\gamma\in K^*$ (the subscript $0$ refers the $0^{th}$ operator associated to $\g$ by the vertex operator). For example, if we identify
\[
\begin{split}
 K=&\{\sqrt{2}(a_0, a_1, \dots,a_8)\mid  a_i\in \ZZ, \sum a_i=0\},\text{ and }\\
\b=&\frac{\sqrt{2}}{3}(1,1,1,1,1,1,-2,-2,-2),
\end{split}
\]
we may take $\gamma= \frac{\sqrt{2}}{9}(1,1,1,1,1,1,1,1,-8)$.

\begin{nota} \label{def:phisubx}
Let
\begin{equation}\label{eE}
 e_M=\frac{1}{16} \omega_M +\frac{1}{32}\sum_{\a\in M(4)} e^\a,
\end{equation}
where $\omega_M$ is the Virasoro element of $V_M$ and $M(4)=\{\a\in
M| \langle\a, \a\rangle = 4\}$.
\end{nota}
It is shown in \cite{DLMN} that $e_M$ is a simple conformal vector of central
charge $1/2$.

Recall that
\[
M^*=\{ \a\in \QQ\otimes_\ZZ M|\, \la \a, \b\ra\in \ZZ \text{ for all
} \b\in M\}= \frac{1}{2} M. \] For $x\in
M^*$, define a $\ZZ$-linear map
\[
\begin{split}
 \la x, \cdot\ra : M &\to \ \, \ZZ_2 \\
y & \mapsto \la x,y\ra  \mod 2.
\end{split}
\]
Clearly the map
\[
\begin{split}
\varphi: M^* & \longrightarrow \mathrm{Hom}_\ZZ(L, \ZZ_2)\\
         x &\longmapsto \la x, \cdot \ra
\end{split}
\]
is a group homomorphism and $Ker \varphi=2 M^* = M$. Hence, we have
\[
\mathrm{Hom}_\ZZ (L, \ZZ_2)\cong M^*/2M^*\cong \frac{1}2 M/M.
\]
For any $x\in M^*=\frac{1}2 M$, $\la x, \cdot\ra $ induces an automorphism $\varphi_x$ of $V_M$
given by
\begin{equation}\label{phi_x}
\varphi_x(u\otimes e^\a) =(-1)^{\la x,\a\ra} u\otimes e^\a \quad
\text{ for } u\in M(1)\text{ and } \a \in M.
\end{equation}

Note that
$$
\varphi_x (e_M)= \frac{1}{16}\omega_M + \frac{1}{32}
\sum_{\a\in M(4)} (-1)^{\langle
x,\a\rangle}(e^\a+\theta(e^\a))
$$
is also a simple conformal vectors of central charge $1/2$.  Since $\varphi_x$
commutes $\theta$,  $\varphi_x (e_M)$ is also contained in $V_{M}^+$.

We call $\varphi_x (e_M)$ a \textsl{conformal vector of central
charge $1/2$ supported at $M$}.

\begin{nota}\label{nota:eandf}
Let
\begin{equation}\label{eandf}
e:=e_M \quad \text{ and } \quad f:=r_M e_M
\end{equation}
and let $U:=\la e, f\ra$ be the subVOA of $V_{M}$
generated by $e$ and $f$.
\end{nota}
\medskip

\begin{rem}\labttr{dihedralsubvoas}
It was shown in \cite{LYY2} and \cite{Sa} that the subVOA $\la e, f\ra$ generated by $e$ and $f$ in $V_{EE_8}$ is isomorphic to
the subVOA in $V^\natural$ generated by the \cvcch associated to the
$2A$ involutions $x$ and $y$. Therefore, we can transfer the study of the dihedral group
$\la x, y\ra$ to the study of some subVOA of $V^\natural$ isomorphic to $\la e, f\ra$.
\end{rem}

\medskip

Next we shall explain how to derive from $e$ and $f$ a pair of $EE_8$-sublattices in a suitable Niemeier lattice, $N$, such that their sum is isometric to $Q=A_2\otimes E_8$. We shall
also embed $U$ into $V_\Lambda^+$ and study the
corresponding Miyamoto involutions in $V_\Lambda^+$, $V_\Lambda$, $V^\natural$,
etc.
We carry out this program for $N=E_8^3$, though it should be possible to do in any Niemeier lattice which  contains a sublattice isometric to $A_2\otimes E_8$.    Such Niemeier lattices are classified in an appendix to this paper.

 \medskip

\subsection{Lie algebra $sl_{n+1}(\CC)$.}\label{sl}

Let $\mathcal{G}=sl_{n+1}(\CC)$ be the simple Lie algebra of type $A_{n}$.
Let $\epsilon_1, \dots, \epsilon_{n+1}$ be an orthonormal basis of $\RR^{n+1}$.
Then the
root lattice system for $\mathcal{G}$ can be identified with
\[
 \{ \epsilon_i -\epsilon_j|\ 0\leq i\neq j\leq n+1\}.
\]

Let $\mathcal{T}$ be the set of all diagonal matrices in $sl_{n+1}(\CC)$ and
denote by $E_{i,j}$ the matrix whose $(i,j)$-th entry is $1$ and all other
entries are zero. Then $\mathcal{T}$ is
a Cartan subalgebra and  the root space for the root
$(\epsilon_i-\epsilon_j), i\neq j$ is $span\{E_{i,j}\}$.

\medskip

Next we shall define several automorphisms of $sl_{n+1}(\CC)$.

Let $\omega=e^{2\pi i/{(n+1)}}$ and denote
\[
     P=
\begin{pmatrix}
 0&1&0 &\cdots &0&0\\
 0&0&1 &\cdots &0&0\\
 &\vdots&  & \ddots& &\vdots\\
  0&0&0 &\cdots  &1&0\\
 0&0&0 &\cdots  &0&1\\
 1&0&0 &\cdots  &0&0
\end{pmatrix}
\]
and
\[
B=\frac{1}{\sqrt{n+1}} [\ \omega^{ij}\ ]_{ 1 \leq i,j\leq n+1}=
\frac{1}{\sqrt{n+1}}
\begin{pmatrix}
\omega& \omega^2 &\cdots &\omega^{n-1}& \omega^{n}&1\\
\omega^2& \omega^4 &\cdots & &\omega^{2n}&1\\
\vdots & & &\ddots &  &\vdots  \\
\omega^{n}& \omega^{n-1}  &\cdots &\omega^2 &\omega&1\\
1&1&\cdots &1&1&1
\end{pmatrix}.
 \]

\begin{de}\label{sAnth}
Define $\tilde{h}_{A_n}: sl_{n+1}(\CC) \to sl_{n+1}(\CC)$  and $s_{A_{n}}
:
sl_{n+1}(\CC) \to
sl_{n+1}(\CC)$ by
\[
 \tilde{h}_{A_{n}}(A) = P^{-1}A P\quad \text{ and } \quad  s_{A_n}( A) =
B^{-1}
A B
\]
for $A\in sl_{n+1}(\CC)$.
\end{de}
Then
\begin{equation}\label{eta}
\tilde{h}_{A_n}(E_{i,j} )= E_{i+1, j+1},
\end{equation}
where $i,j$ are viewed as integers
$mod\,  {(n+1)}$.

Let $\mathcal{C}= \mathcal{G}^{\tilde{h}_{A_n}}$. Then $\mathcal{C}$ is also a
Cartan subalgebra of $\mathcal{G}$.
Note that $\dim(\mathcal{C})=n$ and $\mathcal{C}= \mathcal{G}^{\tilde{h}_{A_n}}$
is
spanned by $P, P^2, \dots, P^{n}$ and
\[
 B^{-1}PB =\mathrm{diag}(\omega, \omega^2, \dots, \omega^n, 1).
\]
Moreover, we have $s_{A_n}(\mathcal{C})=\mathcal{T}$ and
\[
 s_{A_n} \tilde{h}_{A_n}s_{A_n}^{-1} (E_{i,j}) = \omega^{j-i} E_{i,j}.
\]

Let $r_{A_n}:= s_{A_n} \tilde{h}_{A_n}s_{A_n}^{-1}$. Then
\begin{equation}\label{g}
 r_{A_n}= \exp(\frac{2\pi
i}{n+1}(\frac{n}2,\frac{n}2 -1, \cdots, -\frac{n}2+1,-\frac{n}2)_0).
\end{equation}

Define $\theta: \mathcal{G} \to \mathcal{G}$ by
\begin{equation}\label{-at}
 \theta(A)= -A^t, \quad \text{ for } A\in \mathcal{G}.
\end{equation}

By direct computation, we have
\begin{equation}\label{ts-1ts}
  \theta s_{A_n} \theta s_{A_n}^{-1} (A) = (BB^t) A (BB^t)^{-1}
\end{equation}
and
\[
 BB^t=
\begin{pmatrix}
 0&0&0 &\cdots&0&1&0 \\
 0&0&0 &\cdots&1&0&0\\
 \vdots && &\ddots &&&\vdots\\
 0&1&0 &\cdots &0&0&0 \\
 1&0&0 &\cdots&0 &0&0\\
 0&0 &0 &\cdots &0&0&1
\end{pmatrix}
\]

Note that $BB^t$ is symmetric  and it is a permutation matrix of order $2$.

\medskip
Let $h_{A_n}: A_n \to A_n$ be the $\ZZ$-linear map defined by
\begin{equation}\label{hAn}
h_{A_n} ( \epsilon_{i}- \epsilon_{i+1}) = \epsilon_{i+1}- \epsilon_{i+2},
\end{equation}
where $i$ is again viewed as an integer $mod\, (n+1)$. Then $h_{A_n}$ is an
isometry of $A_n$ and it also defines a Coxeter element in $Weyl(A_n)$.

Now identify $(V_{A_n})_1$ with $\mathcal{G}$ by identifying
$(\epsilon_i-\epsilon_j)(-1)\cdot \mathbf{1}$ with $E_{i,i}-E_{j,j}$ and
$e^{\epsilon_i-\epsilon_j}$ with $E_{i,j}$. Then we have
\[
 \tilde{h}_{A_n}( e^\a) =e^{h_{A_n} \a}
\]
by \eqref{eta} and
\[
 \theta(e^\a) = -e^{-\a}
\]
by \eqref{-at}, for any root $\a\in A_n$.

\medskip

\subsection{From $E_8$ to $A_2\otimes E_8$}\label{rank8to24}

In this section, we shall describe how to derive  a pair $EE_8$-sublattices
$M, M'$ in $E_8^3$ such that the subVOA $\la e_M, e_{M'}\ra$ generated by $e_M$
and $e_{M'}$ is isomorphic to $U=\la e,f \ra$ \refpp{eandf}.

Let $L:=E_8\perp E_8$. 
We first show that $L$ contains a sublattice
isometric to $A_2\otimes E_8$.

\begin{nota}\label{hinE8}
Let $h$ be a fixed point free automorphism of $E_8$ of order $3$.
\end{nota}
Set $M=\{ (x,x)\in E_8\perp E_8\mid x\in E_8\}$ and $M'=\{ (hx,x)\mid x\in
E_8\}$. Then both $M$ and $M'$ are isometric to $EE_8$.

\begin{lem}
$M+M'$ is rootless.
\end{lem}

\begin{proof}
Let $(x+hy, x+y)$ be an element of $M+M'$.

If $x+y=0$, then $x+hy= (h-1)y$ has norm $\geq 6$.

If $x+hy=0$, then $x=-hy$ and $ x+y=(1-h)y$ has norm $\geq 6$.

If $x+y\neq 0$ and $x+hy\neq 0$, then  $(x+hy, x+y)\geq 2+2=4$.
\end{proof}

\begin{lem}\labtt{lem:2.3}
$M+M' \cong A_2\otimes E_8$.
\end{lem}

\begin{proof}
Clearly $M'= (h\oplus 1)(M)$
and $(h\oplus 1)$ has order $3$. Since
$h$ is
fixed point free, $M\cap M'=0$.

Since $M+M'$ is rootless, by the $EE_8$-theory established in
\cite{glee8}, $M+M'\cong DIH_{6}(16)\cong A_2\otimes E_8$.
\end{proof}

\begin{nota}\label{QandR}
Set $Q:=M+M'\cong A_2\otimes E_8$ and $R:=\sqrt{3}E_8=EEE_8$.
\end{nota}

\medskip

Now let $A_2\perp A_2\perp A_2\perp A_2$ be a sublattice of $E_8$.

Set $\gamma=(1,0,-1)\in A_2$ and define
\begin{equation}\label{rA2}
 r_{A_2}=\exp(\frac{2\pi i}3 {\gamma}_0)
\end{equation}

\begin{nota}\label{r}
Define $r=r_{A_2}\otimes r_{A_2}\otimes r_{A_2}\otimes
r_{A_2}=\exp(\frac{2\pi i}3 \tilde{\gamma}_0) $ as an automorphism of
$V_{E_8}$, where $\tilde{\gamma}=(\gamma, \gamma, \gamma, \gamma)\in A_2^{\perp
4}$.
\end{nota}

\begin{lem}
 $V_{E_8}^{r}= V_{A_8}$.
\end{lem}

\begin{proof}
Note that the sublattice
\[
 \{\a\in E_8\mid (\a, \tilde{\gamma})\equiv 0 \mod 3\}
\]
is isometric $A_8$.
\end{proof}

Now by \eqref{g}, we have
\[
r_{A_2}= s_{A_2} \tilde{h}_{A_2}s_{A_2}^{-1},
\]
where $s_{A_2}$ and $\tilde{h}_{A_2}$ are defined as before. Thus,
$r:=r_{A_2}\otimes r_{A_2}\otimes r_{A_2}\otimes r_{A_2}$ is
conjugate to
\begin{equation}\label{tildeh}
\tilde{h}:= \tilde{h}_{A_2}\otimes
\tilde{h}_{A_2}\otimes\tilde{h}_{A_2}\otimes \tilde{h}_{A_2}
\end{equation}
in
$Aut(V_{E_8})$. In fact,
\begin{equation}\label{rands}
 r=s \tilde{h} s^{-1},
\end{equation}
where $s: = s_{A_2}\otimes s_{A_2}\otimes s_{A_2}\otimes
s_{A_2}$.

Recall that $\tilde{h}_{A_2}$ induces an element $h_{A_2}\in
Weyl(A_2)$ (cf. \eqref{hAn}).
Thus\\ $h:=(h_{A_2},h_{A_2}, h_{A_2}, h_{A_2})$ defines
an isometry on $E_8$ and it acts fixed point freely on $E_8$.

\medskip

Fix $h$ as above and embed
\[
\begin{split}
 E_8\perp E_8 &\longrightarrow E_8\perp E_8\perp E_8\\
 (\a, \b)&\longmapsto (\a,\b,0)
\end{split}
\]

We shall choose a section of $E_8^3$ in $\hat{E_8^3}$ such that
$e^{(0,0,0)}$ is the identity element of $\hat{E_8^3}$ and $e^{(\a,
\b,\gamma)}=e^{(\a,0,0)}\cdot e^{(0,\b,0)}\cdot e^{(0,0, \gamma)}$, where  $\a,
\b,\gamma\in E_8$ \cite[Chapter 5]{FLM}.

\begin{nota}\label{rhoandeta}
Define $\rho:= r \otimes 1\otimes 1$  and
$\eta=\tilde{h} \otimes
1\otimes 1$ as automorphisms of $V_{E_8^3}\cong V_{E_8}^{\otimes 3}$.
\end{nota}
\begin{lem}
 $\rho$ keeps $V_M$ invariant and $V_M^{\rho}\cong V_{AA_8}.$
\end{lem}

For any even lattice $L$, we define $\theta:V_L\to V_L$ by
\begin{equation}\label{theta}
\begin{split}
 & \theta(\a_1(-n_1)\cdots \a_k(-n_k)\otimes e^\a )\\= &(-1)^k \a_1(-n_1)\cdots
\a_k(-n_k)\otimes((-1)^{\la \a,\a\ra/2} e^\a)
\end{split}
\end{equation}
(cf. \cite{FLM,M1}).
Note that if $L=A_n$ is a root lattice of type $A_n$, by identifying
$(V_L)_1$ with $sl_{n+1}(\CC)$, $(\epsilon_i-\epsilon_j)(-1)$ with $E_{i,i}-E_{j,j}$ and
$e^{\epsilon_i-\epsilon_j}$ with $E_{i,j}$, we have
\[
 \theta|_{sl_{n+1}(\CC)} (A) =- A^t, \qquad A\in sl_{n+1}(\CC)
\]

Now let $e:=e_M$ be a conformal vector in $V_M$ as defined in
\eqref{eE} and  define
$f:=\rho e$.

By the definition of $\theta$, it is clear that
\[
 \theta(e^{(\a,\a,0)})= e^{-(\a,\a,0)} \quad \text{ for all } \a\in E_8
\]
and hence $e$ is fixed by $\theta$.

By \eqref{ts-1ts}, we have
\[
 \theta s_{A_2}\theta s_{A_2}^{-1} = (BB^t) A(BB^t)^{-1},
\]
where $BB^t=\begin{pmatrix}
             0&1&0\\
             1&0&0\\
             0&0&1
            \end{pmatrix}
$ is a permutation matrix of order $2$. Thus, $\theta s\theta s^{-1}$
induces an isometry $\mu: =\overline{\theta s \theta s^{-1}}$ of $E_8$.
This implies
\[
 \theta s\theta s^{-1} (M(1)) \subset M(1)
\]
and
\begin{equation}\label{tsa}
 \theta s \theta s^{-1}(e^\a) =\epsilon(\a) e^{\mu \a}, \quad \text{
for
} \a \in E_8,
\end{equation}
where $\epsilon(\a)=\pm 1$.

\begin{nota}\label{sigma}
Define $\sigma:=s\otimes s\otimes s \in Aut(V_{E_8}^{\otimes ^3})$, considered
as an automorphism of $V_{E_8^3}$.%\bsays{}
\end{nota}
\begin{lem}
 $\theta\sigma\theta\sigma^{-1}(e^{(\a,\a,0)})= e^{(\mu\a,\mu\a,0)}$ for any
$\a\in E_8$.
\end{lem}

\begin{proof}
 By \eqref{tsa},
\[
 \theta\sigma\theta\sigma^{-1}(e^{(\a,\a,0)})= (\epsilon(\a) e^{(\mu \a, 0,0)})
(\epsilon(\a) e^{(0,\mu \a, 0)})= e^{(\mu\a, \mu\a,0)}.
\]
for any $\a \in E_8$.
\end{proof}

Hence,  we have the following corollaries.

\begin{coro}
$\theta\sigma\theta\sigma^{-1}$ fixes $e$.
\end{coro}

\begin{coro}
$\sigma\theta\sigma^{-1}$ fixes $e$.
\end{coro}

\begin{proof}
First we note that $\sigma\theta\sigma^{-1}=
\theta(\theta\sigma\theta\sigma^{-1})$. Since
$\theta$ and $\theta\sigma\theta\sigma^{-1}$
both fix $e$, so does $\sigma\theta\sigma^{-1}$.
\end{proof}

\begin{lem}
 $\sigma^{-1} e$ and $\sigma^{-1} f=\sigma^{-1}\rho  e$ are fixed by
$\theta$.
\end{lem}

\begin{proof}
Since $e$ is fixed by $\sigma\theta\sigma^{-1}$, we have
\[
 \theta\sigma^{-1} e= \sigma^{-1}( \sigma\theta\sigma^{-1} (e)) = \sigma^{-1}e.
\]
Moreover,
\[
\begin{split}
 \theta\sigma^{-1} {\rho} e &= \theta {\eta}\sigma^{-1} e\qquad
(\text{since } {\rho}=\sigma {\eta}\sigma^{-1})\\
& = {\eta}\theta\sigma^{-1}e \qquad (\text{since }
\theta{\eta}={\eta}\theta)\\
& = {\eta}\sigma^{-1}e \\
&= \sigma^{-1}{\rho} e.
\end{split}
\]
Thus, $\sigma^{-1} f$ is fixed by $\theta$.
\end{proof}

\begin{lem}
Set $e'=\sigma^{-1} e$ and $f'=\sigma^{-1} f$. Then
$ e'\in V_M^+$ and $f'\in V_{M'}^+$ and hence $e= \varphi_x e_M$ and
$f'=\varphi_y e_{M'}$ for some $x\in \frac{1}2 M$ and $y\in \frac{1}2M'$, where
$\varphi_x$ and $\varphi_y$ are defined as in Notation \ref{def:phisubx}.
\end{lem}

\begin{proof}
 Since $\sigma^{-1}$ keeps $V_M$ invariant, we have $\sigma^{-1} e\in V_M$ and
thus $e'\in V_M^+$ as $\sigma^{-1} e $ is fixed by $\theta$.

On the other hand, ${\eta}$ maps $V_M$ to $V_{M'}$. Therefore,
\[
 f'=\sigma^{-1} {\rho} e ={\eta} \sigma^{-1}e \in V_{M'}
\]
and thus $f'\in V_{M'}^+$.

Next we note that $\tau_e$ acts on $(V_{M})_1=(M(1))_1$ as $-1$. Thus,
$\tau_{\sigma^{-1}e} =\sigma\tau_e\sigma^{-1}$ also acts as $-1$ on $(V_M)_1$.
Now by the classification of conformal vectors of
central charge $1/2$ in $V_{EE_8}^+$ (cf. \cite{O+102,LSY}), we have
$\sigma^{-1} e= \varphi_x e_M$ for some $x\in E_8$. Similarly, we have
$f'=\varphi_y e_{M'}$ for some $y\in E_8$.
\end{proof}

\begin{thm}\label{embedofL}
The Leech lattice $\Lambda$ contains a sublattice isometric to $A_2\otimes E_8$
and hence  $U:=\la e , f \ra$,   the subVOA generated by $e$
and $f$, can be embedded into $V_\Lambda^+$.
\end{thm}

\begin{proof}
An explicit embedding of $A_2\otimes E_8$ into $\L$ can be found in
Appendix of \cite{glee8}. Thus,
\[
U\cong \sigma^{-1} U \subset V_{A_2\otimes E_8}^+ \subset V_{\Lambda}^+
\]
as desired.
\end{proof}

\begin{rem}\labttr{qinleech}
One can also obtain an embedding of $Q\cong A_2\otimes E_8$ into $\Lambda$ as
follows: Let $h\in O(\L)$ such that $h$ has order $3$ and trace $0$. The fixed
point sublattice of $h$ in $\L$ is isometric to $R\cong \sqrt{3}E_8$ and the
annihilator of $R$ in $\L$ is
\[
 ann_\L(R)\cong Q=A_2\otimes E_8.
\]
Recall that $N_{O(\L)}(h)\cong Sym_3\times 2\cdot Alt_9$ in this case \cite{Atlas}.
\end{rem}

\begin{rem}\labttr{dihmm'}
Since $\rho$ is conjugate to $\eta$ in $Aut(V_{E_8^3})$, it is clear that the
subVOA $\la e_M, \rho e_M\ra \cong \la e_M, \eta e_M\ra$.
%\bsays{this is a bit confusing since this U is not the U of the previous Theorem}
Note also that
$\eta e_M\in V_{M'}$ is a \cvcch supported at
$M'$. Thus, we may study the
properties of the dihedral group $\la \tau_e, \tau_f\ra$ in $Aut(V_\L^+)$ or
$Aut(V^\natural)$ by examining the configuration $(M , M')$ in $\L$.
\end{rem}

\section{Overlattices and gluing}

The goal
is to discuss overlattices for $Q\perp R$ which are isometric to
$\L$, the Leech lattice.  We explain how $Q\perp R$ is contained in
a copy of $E_8^3$ and $\L$ in such a way that the common stabilizer
is a group $2{\cdot }Alt_9$ and triality of groups of type $D_4$ is
involved.

Our argument uses triality to prove existence of a Leech lattice and
explain the occurrence of the group $2{\cdot }Alt_9$ as the
stabilizer of a relevant gluing map. We shall give an easy proof
that  $2{\cdot}Alt_9$ occurs in a gluing based on existence of a
Leech lattice in the appendix.

\medskip

We discuss the following situation.

\begin{nota}\labttr{orbitsgluemaps}
We fix an orthogonal direct sum of integral lattices, $Q\perp R$.
Suppose that an index $m$ is given and that we are to study the set
$\frak X :=\{L\mid Q\perp R\le L \le \dual Q \perp \dual R,
|L:Q\perp R|=m, L\cap \QQ\otimes Q = Q, L \cap \QQ \otimes R = R
\}$. We wish to understand the orbits of $O(Q)\times O(R)$ on $\frak
X$.  Let $\frak Y :=\{ L \in \frak X \mid L \text{ is integral }\}$.
\end{nota}

\begin{nota}\labttr{triples} We define
$$\frak T:=\{ (A, B, \psi ) \mid A \text{ is a subgroup of order $m$ in $\dg Q$},$$
$$B \text{ is a subgroup of order $m$ in } \dg R, $$
$$\psi \text{ is an isomorphism of
$A$ to $B$ } \}.$$
\end{nota}

\begin{prop}\labtt{x1} (i) $\frak X$ is in bijection with the set of triples $\frak T$.

(ii) $L$ is integral if and only if $\{ (a, \psia ) \mid a \in A \}$
is a totally singular subspace of the quadratic space $\dg Q \perp
\dg R$ with natural $\QQ/\ZZ$-valued bilinear  form.

(iii) The totally singular condition holds if and only if for all
$a\in A$,
$(a,a)+(\psia, \psia )=0 \in \QQ/\ZZ$.  In particular,
there exists a scalar so that $\psi$ is a scaled isometry.
\end{prop}

Special case: the spaces $\dg Q$ and $\dg R$ have a scaled isometry,
e.g. $Q=A_2 \otimes E_8$ and $R=EEE_8$.

\begin{de}\labttr{actiononmaps}  A group action is assumed to be on the left.
Suppose that the group $G$ acts on the set $A$ and the group $H$
acts on the set $B$. We have an action of $G\times H$ on $Maps(A,B)$
as follows.  If $f$ is a map, then $(g,h)\cdot f$ is the map which
takes
$a$ to $h(f(g^{-1}a))$. 
\end{de}

\begin{de}\labttr{similitude}
A similitude is a linear map between quadratic spaces which is a
scaled isometry.  The set of self-similitudes of a quadratic space
is a group which contains the orthogonal group as a normal subgroup.
\end{de}

Now let $G_Q$ be the group of similitudes on $\dg Q$ and $G_R$ the
group of similitudes on $\dg R$. Let $Z=Q$ or $R$. For $g$ in one of
these groups $G_Z$, define $\l (g)$ to be the scaling factor, i.e.,
the nonzero scalar such that $\l (g)\cdot (x,y)=(gx, gy)$ for all
$x, y\in \dg Z$.

The above definition gives an action of $G_Q \times G_R$ on $\frak
X$. The subgroup $G_{Q,R}:=\{ (g,g')\in G_Q\times G_R \mid \l (g)=\l
(g') \}$ is the stabilizer in $G_Q\times G_R$ of the condition
$(a,a)+(\psia, \psia )=0 \in \QQ/\ZZ$ in \refpp{x1}(ii) and of the
set $\frak Y$.

\section{The subgroup $2^2{\cdot}O^+(8,2)$ of $2^2{\cdot }O^+(8,3)$}

\def\weyleh{{Weyl({E_8})}}

The structure of $2{\cdot}O^+(8,2)\cong Weyl(E_8)$ is well known.
It embeds in $O^+(8,3)$ as a subgroup generated by reflections.  One
gets such an embedding by taking the $E_8$ lattice modulo 3 with the
associated quadratic form.

The group $O^+(8,3)$ has the property that its second derived group
has index 8, is a perfect central extension of $\Omega ^+(8,3)$ and
gives the quotient $Dih_8$.

Its order is therefore $2^{15}3^{12}5{\cdot 7}{\cdot 13}$.  It
contains $\weyleh$ with index $2{\cdot 3^7}13$.

We need a few standard facts. For all $q$, the group $\Omega^+(8,q)$
has a group of graph automorphisms isomorphic to $Sym_3$.  This
group acts faithfully on the Schur multiplier when this is
isomorphic to $2\times 2$, i.e., for $q=2$ and $q$ odd.

\begin{lem}\labtt{invollift}
Let $F$ be a field of characteristic not 2 and $n
\ge 2$.  An involution in $SO(n,F)$ lifts to an element of order 2 or 4 in $Spin(n,F)$.  It lifts to an element of order 4 if and only if the multiplicity of $-1$ in its spectrum on the natural $n$-dimensional module is $2(mod \, 4)$.
\end{lem}
\pf  This is a standard fact.  A proof may be found in \cite{grelab}. \eop

We have $X:=Weyl(E_8)'/\{\pm 1\}\cong \Omega ^+(8,2)$.  There are
three conjugacy classes of maximal parabolic subgroups with Levi
factors of type $A_3$.  Let $P_i$ be representatives, $i=1,2,3$. For
each $i$, $P_i$ lifts in the covering group $\widehat X$ to a group
$Q_i$ of the shape $(2\times 2_+^{1+6})GL(4,2)$.  In a quotient of
$\widehat X$ by a group of order 2, two of these $Q_i/Z$ are
isomorphic to $2_+^{1+6}GL(4,2)$ and the other is isomorphic to
$2^7{:}GL(4,2)$.

\subsubsection{Creating double covers of $Sym_9$ in $W$ with triality}

\begin{prop}\labtt{graphd4}
Let $X_1 < X_2$ be a containment of perfect groups isomorphic to
$2^2{\cdot}\Omega^+(8,2)$ and $Spin^+(8,3)$ respectively.

There exists a subgroup $\Sigma\cong Sym_3$ of $Aut(X_2)$ which
complements $Inn(X_2)$ and such that $\Sigma$ stabilizes $X_1$.
\end{prop}
\pf Let $r$ be an element in $Aut(X_2)$ corresponding to a
reflection in a representation $\rho$ of $X_2$ on its natural
quadratic space $V:=\FF_3^8$. We assume that $r$ normalizes $\rho
(X_1)$ and so $\la \rho  (X_1), r \ra \cong Weyl(E_8)$. We extend
$\rho $ to a representation of the semidirect product $X_2 \la r
\ra$.

Let $h$ be an automorphism of order a power of 3 which is outer and
is inverted by $r$ under conjugation.

We consider an arbitrary representation $\s$ of $X_1$ on the
quadratic space   $V$ such that the kernel of $\s$ has order 2.

It has the property that exactly one of the three conjugacy classes
of maximal  parabolic subgroups of $X_1$ with Levi factor of type
$A_3$ acts by $\rho$ as a monomial group $2^7{:}Alt_8$ (we use the
term {\it parabolic} for a subgroup of $X_1$ if it contains $Z(X_1)$
and maps modulo $Z(X_1)$ to a parabolic of the group of Lie type
$X_1/Z(X_1)$). Let $P$ be such a maximal parabolic.

Then, $\s (P)$  can be conjugated by an element of $\rho (X_2)$ to
$\rho (Q)$, where $Q$ is a parabolic subgroup of $X_1$ such that
$\rho (Q)$ acts monomially with respect to some basis, say $\mathcal
A$ of $V$. We  may assume that $r$ is chosen to normalize $Q$. Our
hypotheses imply that $\rho ( \la Q, r \ra )$ is a uniquely
determined index 2 subgroup of the full orthogonal monomial group on
$\mathcal A$.

The group $\rho (\la X_1, r \ra)$ is generated by $\rho  ( \la Q, r
\ra )$ together with a product $r_1r_2$ of commuting reflections,
one of which, say $r_1$, is a reflection at $\pm b \pm b'$, for some
$b, b' \in \mathcal A$. The other reflection, $r_2$, may be taken as
reflection at some element $s$ of the quadratic space which has the
property that for all $b\in \mathcal A$, $(s, b)\in \{ -1, 1\} (mod
\, 3)$.   It is clear that any two such  $s$ are in the same orbit
under the monomial group on $\mathcal A$.

We apply above remarks to the composition $\s = \rho h$.  It follows
that there exists $g\in X_2\la r \ra$ so that $\rho (g) \rho ( h
(X_1))\rho (g)^{-1}=\rho (X_1)$. Let $i_g\in Aut(X_2)$ be
conjugation by $g$. It follows that $i_gh \in Aut(X_2)$ takes $X_1$
to itself and induces a group of order 3 on $Z(X_1)=Z(X_2)$. This
proves the result since $\la \{ i_k\mid k \in X_1 \}, r, i_gh \ra
\cong Aut(X_2)$. \eop

\begin{prop}\labtt{2alt9}
We use the notation of \refpp{graphd4} and its proof. Let $X$ be a
subgroup of $X_1$ so that $\rho (X) \cong Alt_9$. Let $\a\in \Sigma$
so that $\a$ does lie in the group $Inn(X_2)\la r \ra$. Then  $\rho
(\a (X)) \cong 2{\cdot }Alt_9$, the covering group of $Alt_9$.
\end{prop}
\pf The hypotheses on $\a$ imply that $\a$ does not stabilize the
subgroup $Ker(\rho )\cap X\cong 2$. Therefore, the image of $\a (X)$
in $\rho (X_2)$ is isomorphic to $2{\cdot }Alt_9$ \refpp{invollift}.
\eop

\begin{nota}\labttr{a8e8mod3}
Let $L=E_8$ and $\Phi$ the root system.  Let $\Phi_0$ be a sub root system of type $A_8$, $\Phi_0 \subset \Phi$.
Let $W$ be the Weyl group of $\Phi$ and let $W_0$ be the Weyl group of $\Phi_0$. Then $W_0\cong Sym_9$ and its action on
$L/3L$ has constituents of dimensions 1 and 7.  There are submodules of these
dimensions and
each is nonsingular.
\end{nota}

It is straightforward to check the last two statements above with a standard model of the relevant root lattices.

\begin{nota}\labttr{a8e8mod3nota}
Let $q$ be the  reflection at the nonsingular 1-dimensional module
described in \refpp{a8e8mod3}. We therefore have the subgroup $\pi
(W_0) \times \la q \ra \cong 2 \times Sym_9$ of $O(\dg Q) \times
O(\dg R)$. Its commutator subgroup is isomorphic to $Alt_9$ and the
commutator quotient is $2\times 2$. The procedure of \refpp{graphd4}
and \refpp{2alt9} gives a subgroup $K_0$ of $G_0 \cong Spin^+(8,3)$
so that $K:=K_0'$ satisfies $K\cong 2{\cdot }Alt_9$ and $K_0/K\cong
2$. Thus, $K_0$ is a covering group of $Sym_9$ (there are two such
covering groups, depending on whether a transposition is represented
by an element of order 2 or 4).
\end{nota}

\begin{lem}\labtt{maxa9weyle8}
(i)
The group $Z(W) \times W_0$ is maximal in $W$.

(ii)
The group $Z(W) \times W_0'$ is maximal in $W'$.
\end{lem}
\pf
(i) This follows from the classification of root systems.

(ii) Since $Z(W) \times W_0'$ does not contain reflections, this is
more difficult. By use of $Aut(W/Z(W))$, we see that the proof is
equivalent to proving that $K$ is maximal in $W'$, where $K\cong
2{\cdot }Alt_9$ is the group created in Proposition \ref{2alt9}.

We let $\a $ be a root and $X:=Stab_W(\a)\cong 2 \times Sp(6,2)$, a group of order $2^{10}3^45{\cdot }7$.

Suppose that there is a subgroup $S$ so that $K< S <W'$.  Define
$T:=Stab_S(\a), T_0:=Stab_{K}(\a)\cong SL(2,8){:}3$. We have
$|S:K|=|T:T_0|$. By \refpp{maxl283sp62}, $TZ(X)=T_0Z(X)$ or $X$.
Since $T_0$ and $X$ are generated by their odd order elements and
$Z(X)$ is a 2-group, $T=(T\cap Z(X))\times (T\cap X')$.  The left
factor has order 1 or 2 and the right factor is $T_0$ or $X'$.

If $T\cap X'=T_0$, either $K=S$, which is impossible, or $|S:K|=2$, which would mean that
$K$ is normal in $S$.  But this would mean that
$W_0'$ is contained in $N_W(W_0')$ with index divisible by 4.  This is clearly impossible since $W_0$ is self-normalizing in $W$.

We conclude that $T\cap X'=X'$.   This means that
$S$ has index 1 or 2 in $W'$, which is a perfect group.  Therefore $S=W'$, a contradiction.
\eop

\begin{nota}\labttr{h}
 We define the group $H$ to be a natural $2{\cdot }Alt_8$ subgroup of
$K$ where $K\cong 2{\cdot} Alt_9$ is the group defined in Proposition
\ref{2alt9}.
\end{nota}

\begin{lem}\labtt{htransroots}
The group $H$ acts transitively on roots.  A stabilizer has the form
$2^3{:}7{:}3$. For the action of $K$ on roots, a stabilizer has the
form $SL(2,8){:}3$.
\end{lem}
\pf We start with the Barnes-Wall viewpoint for $E_8\cong BW_3$.
Consider a standard frame $F$ of minimal vectors. In the BRW group,
$G$, $Stab_G(F)\cong 2^{1+6}2^3.GL(3,2)$ and for $\a \in F$,
$J:=Stab_G(\a )$ has the form $2^3.2^3.GL(3,2)$.

We may replace $H$ by a conjugate to assume that its intersection
with $J$ contains  a group of the form $7{:}3$. The intersection has
order bounded below by $8!/240=2^3.3.7$. If the intersection were
larger, it would have order of the form $2^a3{\cdot }7$, for some
$a\ge 3$.  By Sylow 7-theory, $a$ is divisible by 3, whence $a=6$.
Thus, the intersection would contain a maximal subgroup of a Sylow
2-group $P$ of $H$ which meets $Z(H)$ trivially.  This is impossible
by group transfer theory (since $Z(H) \le H'$ implies $Z(H)\le P'$).
It follows that the stabilizer order is exactly $8!/240=2^3.3.7$.
Transitivity follows. Finally, we argue that a stabilizer, $S$, has
the form $2^3{:}7{:}3$.  Since $S$ is contained in a group of the
form $2^3.2^3.GL(3,2)$, if the statement is false, $S\cong GL(3,2)$.

In the stabilizer of a root, there is up to conjugacy just one Sylow
7-subgroup and up to conjugacy just two subgroups of the form
$2^3{:}7$, because the action of a group of order 7 on $O_2(J)$ is
completely reducible with two non-isomorphic (and dual)
constituents.  Each constituent has order $2^3$.  One constituent is
represented by $O_2(G)\cap J$.  The second constituent is
represented by a subgroup of the group $A$ of permutation matrices
in $G$, which is isomorphic to $AGL(3,2)$.  It is clear that
$O_2(A)$ fixes the root $(\half ^8)$, in the standard $E_8$
notation.  It follows that $S$ has the form $2^3{:}7{:}3$, rather
than $GL(3,2)$.

Now let $T$ be the stabilizer of $\a$ in $K$.  Then $|Y:S|=9$. Thus,
$T$ is a triply transitive group of degree 9.  By a classification
\cite{suzuki2closed}, $T\cong SL(2,8){:}3$. \eop

\section{ Some properties of $Q\perp R$}

\def\oweh{|Weyl({E_8})|}
\def\ooht{|O^+(8,3)|}
\def\oetwo{O^+(8,2)}
\def\oethree{O^+(8,3)}

\begin{lem}\labtt{minveca2e8}
The minimal vectors in $Y:=A_2\otimes E_8$ have norm 4 and are
expressed as the union of the three sets $\a \otimes \Psi$, where
$\a$ runs over three pairwise nonproportional vectors of the
$A_2$-factor and $\Psi$ is the set of roots for the second factor.

(i) These three sets are maximal sets of pairwise doubly even sets
(i.e. $(x, y)\in 2\ZZ$ for all $x, y$ in the set) of minimal
vectors;

(ii)  A doubly even set of minimal vectors of cardinality at least
240 equals one of these sets. In particular, a doubly even set of
minimal vectors which meets every coset of $3Y$ in $Y$ which
contains a minimal vector is one of the above sets.

(iii) These sets are permuted by the isometry group of the lattice.
We have $O(Y)=U\times T$, where $T$ acts on  each $span_{\ZZ}(\a
\otimes \Psi)$ as its full isometry group, isomorphic to $O(E_8)$,
and where $U\cong Sym_3$ permutes the three sets $\a \otimes \Psi$.
We may take three nonproportional vectors $\a_1, \a_2, \a_3$ whose
sum is 0 and choose  identifications $U\cong Sym_3$ and
$Y=A_2\otimes E_8$ so that the permutation $p$ corresponds to the
isometry $p(\a_i \otimes x)=\a_{p(i)}\otimes x$, for all $x \in
E_8$.
\end{lem}
\pf (i) Let $\Delta$ be the set of roots of the first factor.

Choose a single minimal vector, say $\a \otimes \g$.  The set of
norm 4 vectors which have even inner product with it is $E:= (\a
\otimes \Psi) \cup (\Delta \otimes \g) $.  The set of elements of
$E$ which have even inner product with every element of $E$ is just
$\a \otimes \Psi$ and any $\b \otimes \g$ has odd inner product with
at least one member of $\a \otimes \Psi$. If follows that $\a
\otimes \Psi$ is a doubly even set, maximal under containment.

(ii) The second statement follows from the first, which we now
prove.

Suppose $S$ is a doubly even set of minimal vectors with $|S|\ge
240$.    Let $S$ be the union of sets $\a \otimes P, \b\otimes Q, \g
\otimes R$, where $\a, \b , \g$ are pairwise nonproportional vectors
in $\Delta$. We want to prove that $S$ is one of these.  Suppose
that this is not so.  Then none of $P, Q, R$ equals $\Psi$ and at
least one of them, say $P$,  has cardinality at least $240/3=80$,
which means that $P$ represents at least 40 nonsingular cosets of
$E_8$ mod 2.   Therefore, the span of $P+2E_8$ has dimension $d\ge
6$.   Since $\b \otimes Q$ has even inner product with $\a \otimes
P$, $(P, Q)\le 2\ZZ$.  Therefore $Q$ represents nonsingular cosets
in the annihilator space of the above  span of $P+2E_8$.  This
annihilator space has dimension $8-d$, so $|Q|\le 6$.  Similarly,
$|R|\le 6$.   If $Q\ne \emptyset$, then $|P|\le 126$ and so $|P\cup
Q\cup R|\le 126+6+6< 240$,  a contradiction to $|S|\geq 240$.  We
conclude that $Q=R=\emptyset$.

(iii) This follows from the characterization of (ii).  The obvious
map $O(Y)\rightarrow Sym_3 \times O(E_8)$ is an isomorphism of
groups. \eop

\begin{lem}\labtt{trivialmodp}
Suppose that $A$ is a free abelian group and that $n>1$ so that the
finite order automorphism $g\ne 1$ acts trivially on $A/nA$.   Then
$n=2$, $g$ has order 2 and $A$ is the direct sum of $A^+:=\{a\in
A\mid ga=a \}$ and $A^-:=\{a\in A\mid ga=-a \}$.
\end{lem}
\pf Suppose that $g$ has order $p^a>2$ for a prime number $p$ and
integer $a\ge 1$. There exists a direct summand $B$ of $A$ so that
on $B$, the minimum polynomial of $g$ is the cyclotomic polynomial
$\Phi_{p^a}$ of degree $p^a-p^{a-1}$. In the ring of integers, $\ZZ
[{\root {p^a} \of 1}]$, if $\pi$ is a primitive $p^a$-th root of 1,
then $\pi^{p^a-p^{a-1}}$ generates the ideal $p\ZZ [{\root {p^a} \of
1}]$ \cite{weiss}.

It follows that if $g$, of arbitrary finite order greater than 1,
acts trivially on $B/nB$, then $n$ is a power of $p$ and
$p^a-p^{a-1}=1$, whence $p^a=2$.

We therefore may assume that $g$ has order 2 and  $n=2^f$, for some
$f\ge 1$. In this case $g$ acts trivially on $A/2A$.  If we prove
that the decomposition $A=A^+\oplus A^-$ holds, then $f\le 1$
follows (since $A^-\ne 0$).  We may therefore assume that $n=2$.

There exists an endomorphism $E$ of $A$ so that $g=1+2E$.  Then
$1=g^2=1+4(E+E^2)$, whence $E(E+1)=0$ in $End(A)$. For $a \in A$, we
have $g(Ea)=(1+2E)Ea=(E+2E^2)a=-Ea$ and
$g(E+1)a=(1+2E)(E+1)a=(2E^2+3E+1)a=(E+1)a$, so $a=(E+1)a-Ea\in
A^++A^-$. \eop

\begin{lem}\labtt{stab0}
Suppose that $L$ is an overlattice of $Q\perp R$ such that $L\cap
\QQ Q=Q$, $L \cap \QQ R=R$ and $L$ is stable under $O_3(O(Q))$.
Write $O(Q)=X\times Y$, where $X\cong Sym_3$ and $Y \cong O(E_8)$
\refpp{minveca2e8}. Then $Stab_{X\times Z(Y)}(L)\cong Sym_3$ and
$C_{X\times Y}(Stab_{X\times Z(Y)}(L))=Y$, the subgroup of $X\times
Y$ which fixes each of the sets $\a \otimes \Phi$.
\end{lem}
\pf Let $h$ generate $O_3(Q)$.  Since $h$ acts trivially on $R$, $h$
acts trivially on $Proj_Q(L)/Q$, which means $(h-1)Proj_Q(L)\le Q$.
Since $O(Q)=X\times Y$, where $X\cong Sym_3$ and $Y \cong O(E_8)$
\refpp{minveca2e8}, the fact that $\dg Q$ is an absolutely
irreducible module for $Y$ means that elements of $X$ act as scalars
on $\dg R$. If an involution $t$ of $X\setminus Z(X)$, acts as the
scalar $c\in \{\pm 1\}$ on $\dg R$, there exists $z\in Z(Y)$ which
acts on $\dg Q$ as $c$. We have $Stab_{X\times Z(Y)}(L)=O_3(X\times
Z(Y))\la tz\ra \cong Sym_3$.  The last statement is clear. \eop

\begin{prop}\labtt{stab1}
Suppose that $L$ is a Niemeier lattice and that
$\th \in O(L)$ has order 3 and satisfies
$L^+ (\th) \cong R$ and $L_+(\th )
\cong Q$.
Then the stabilizer in $O(L)$ of the gluing map for $L$ over $Q \perp R$
is $N_{O(L)}(\la \th \ra)$.
\end{prop}
\pf
Note that both $L^+ (\th)$ and $L_+ (\th)$ are direct summands of $L$.
Let $\a$ be the gluing map and $S$ its stabilizer in $O(L)$, i.e.,
$\{g\in O(L) \mid g(Q)=Q, g(L)=L, g\circ \a = \a\}$.
If $g\in N_{O(L)}(\la \th \ra)$, clearly $g$ fixes both $L^+(\th )$
and its annihilator $L_+(\th )$.
Since it fixes $L$ and commutes with projections, it fixes the gluing map.

Now, we prove that $S\le N_{O(L)}(\la \th \ra)$. Since $S$ acts on
$Q$, it permutes the set $F$ of norm 4 vectors.  There is a
partition of $F$ into three sets $F_i$, $i=1, 2, 3$ so that $Q_i$,
the $\ZZ$-span of $F_i$,  is an $EE_8$ lattice in which $F_i$ is the
set of minimal vectors. It follows from \refpp{minveca2e8} that $S$
permutes these three sets.

We now refer to the notation of \refpp{stab0}.  Since
$Stab_{X\times Z(Y)}(L)$ acts on
$\{F_1, F_2, F_3\}$ as $Sym_3$, $S=Stab_{X\times Z(Y)}(L)T$ where $T$ is the subgroup
of $S$ which normalizes each of $Q_1, Q_2, Q_3$.  By \refpp{stab0}, $T\le C(Stab_{X\times Z(Y)}(L))$.
It follows that $S\le N_{O(L)}(\la \th \ra)$.
\eop

\section{Overlattices of $Q\perp R$}

We continue to use the notations $Q, R$ (Notation \ref{QandR}.  This
section will explain which $L$ may arise in \refpp{stab1}.

\begin{lem}\labtt{qrine8e8e8}
There exist embeddings of  $Q \perp R$ in $E_8^3$.  In fact, there are at least two kinds of embeddings.

(i) (3-cycle type)  there exist embeddings such that $R$ is the fixed point sublattice of an automorphism of order 3 which permutes the three direct summands cyclically; and

(ii)  (1+2 type)  there exist embeddings such that $\QQ R \cap E_8^3$ is an orthogonal direct summand of $E_8^3$.
\end{lem}
\pf (i) is trivial.  Compare Appendix \ref{QRinN}. (ii) Let $A, B,
C$ be the three indecomposable summands of a lattice $L$ isometric
to $E_8^3$. Fix an isometry $\phi : A\rightarrow B$. We take
faithful actions of $O(L)$ on $A, B, C$ for which $\phi$ is an
invariant map.

There is an automorphism $h$ of order 3 of $E_8$ which does not have eigenvalue 1.  Then, the endomorphism $h-1$ triples norms.

So, $(h-1)C\cong R$. Now, define a lattice $J\le A\perp B$ by
$J:=(h-1)A+(h-1)B+K$, where $K:=\{ a+\phi (a) \mid a \in A \}\cong
EE_8$. Then $J$ is the sum of $K$ and $K':=\{ a+\phi (ha) \mid a \in
A \}\cong EE_8$. It is easy to prove that $J$ is rootless  (since
any element of $A+B$ of the form $a+b$, where $a\ne 0, b\ne 0$ has
norm at least 4).  Since $K\cap K'=0$, the classification
\cite{glee8} identifies $J$ as isometric to $A_2\otimes E_8$. \eop

\begin{nota}\labttr{gluealpha}
We fix an overlattice of $Q\perp R$ which is isometric to $E_8^3$ and is of type (i) in \refpp{qrine8e8e8}.  Let $\a$ be the associated gluing map, of $\dg Q$ with $\dg R$.  If $\g$ is any gluing map, let $L_{\g}$ be the overlattice associated to it.  So, $L_{\a}=L$ is our initial choice of $E_8^3$-overlattice.
\end{nota}

We seek a new gluing map which gives a rootless Niemeier lattice.  Such a lattice would be isometric to the Leech lattice, by a well-known classification.

\begin{nota}\labttr{pi}
Let $\pi $ be the representation of $O(Q)\times O(R)$ on
$\dg Q \perp  \dg R$.  The  image of $\pi$ lies in $O(\dg Q) \times O(\dg R)$.
\end{nota}

\subsection{The new gluing map}

We look for a similitude on our quadratic spaces which respects a subgroup isomorphic to $2{\cdot }Alt_9$ and defines a Leech overlattice.

\begin{lem}\labtt{maxl283sp62}
Suppose that $A\cong SL(2,8){:}3$ and $A\le B\cong Sp(6,2)$.  Then $A$ is a maximal subgroup.
\end{lem}
\pf
The embedding is essentially unique,
by the 2-modular representation theory of $SL(2,8)$.
A Sylow 7-normalizer in $A$ is a Sylow 7-normalizer in $B$.   We have
$|B:A|=2^6{\cdot}3{\cdot 5}$.  The only divisors of this which are $1(mod \, 7)$ are products of a subset of $2^3, 2^3, 3{\cdot 5}$.
Now suppose that $S$ is a subgroup, $A < S < B$.
If $|B:S|\le 15$, we have a contradiction since $Sp(6,2)$ does not embed in $Sym_{15}$.
Therefore,
$|S:A|\le 15$.
If $|S:A|=8$, then $A'\cong SL(2,8)$ is normal in $S$,
which is impossible by above Sylow 7-theory.  We conclude that
$|S:A|=15$ and $|B:A|=2^6$.
Therefore, in the action of $S$ on the left cosets of $A$,
$A$ fixes 6 cosets
and has a single orbit of length 9.
We now use the fact that if $T\in Syl_2(A)\subset Syl_2(S)$,
$N_S(T)$ operates transitively on the fixed points of $T$.
Here, $T$ fixes 7 of the 15 points.  A group of order 7 in the normalizer acts by a 7-cycle on the $A$-orbit of length 9 and trivially on the other seven points. This is a contradiction.
Therefore, $S$ does not exist.
\eop

\begin{nota}\labttr{newglue}
We take the groups $K_0$ and $K$
constructed in
\refpp{a8e8mod3nota}.
Let $u\in K_0 \setminus K$.

We define a new glue map by conjugating with $u$: $\b :
x+Q\rightarrow u(\a (u^{-1}(x+Q)))$, for $x\in \dual Q$.
\end{nota}
The stabilizer of $\b$  in $W$ is $W\cap uWu^{-1}$, which contains $K$.
(Actually, $K$ is a maximal subgroup of $W'$ \refpp{maxa9weyle8} and $W'$ is the only maximal subgroup of $W$ which contains $K$).

\begin{nota} \labttr{mphi}
If $Q$ has minimal vectors $\a \otimes \g$, for roots $\a \in A_2$ and $\g \in E_8$, then the minimum norm vectors in $\dual Q$ are the norm $\frac 43$ vectors of the form
$\frac 13 (\a' - \a'')\otimes \g$, where
$\a, \a', \a''$ are roots in $A_2$ such that an isometry of order 3 takes $\a \mapsto \a' \mapsto \a''$, and where $\g$ is a root of $E_8$.
The minimum norm vectors in $\dual R$ are the vectors of the form $\third \d$, where $\d$ is one of the 240 minimal vectors in $R\cong \sqrt 3 E_8$.
Take any isomorphism  $\phi : E_8 \rightarrow R$ which triples norms.
Note that the vectors $(\a'-\a'')\otimes \g$ and $(\a'-\a'')\otimes \g'$ have inner product
$-(\g, \g')$.  Thus the three sets
$\{ \third (\a-\a')\otimes \g +  \phi (\third \g ) \mid  \g \in E_8\}$,
$\{ \third (\a'-\a'')\otimes \g +  \phi ( \third \g ) \mid  \g \in E_8\}$ and
$\{ \third (\a''-\a')\otimes \g +  \phi (\third \g ) \mid  \g \in E_8\}$,
are pairwise orthogonal root systems of type $E_8$.Define $M(\phi )$ to be the overlattice of $Q \perp R$ which is the $\ZZ$-span of these three sets.  Then $M(\phi )\cong E_8^3$.  Any overlattice of $Q\perp R$ which is isometric to $E_8^3$ equals one of these $M(\phi )$, of which there are
$|O(E_8)|=2^{14}3^5 5^2 7$.
\end{nota}

\begin{lem}\labtt{ww}
(i) The action of $\pi (O(Q)\times O(R)) \cong W \times W$ is transitive on the set of $M(\phi )$.  In fact, the action of either direct factor,
$\pi (O(Q))$ or $\pi (O(R))$, is regular.

(ii) The stabilizer of a given $M(\phi )$ in
$O(\dg Q) \times O(\dg R)$
is a
diagonal subgroup of
$\pi (O(Q)\times O(R))$.

(iii) Any element of $O(\dg Q) \times O(\dg R)$ which moves one $M(\phi )$ to another is in $\pi (O(Q)\times O(R))$.
\end{lem}
\pf
(i) Straightforward.

(ii) Let $S$ be the stabilizer of $M(\phi )$.  Then $S$ acts faithfully on both $\dg Q$ and $\dg R$.
The conclusion follows.

(iii) Suppose that
$g\in O(\dg Q) \times O(\dg R)$ moves one $M(\phi )$ to another, say $M(\phi ')$.
By the transitivity result of (i), there exists $h\in \pi (O(Q)\times O(R))$ which takes $M(\phi )$ to $M(\phi ')$.
Then $h^{-1}g$ stabilizes $M(\phi)$, so $h^{-1}g \in \pi (O(Q)\times O(R))$, by (ii).
\eop

\begin{lem}\labtt{lbetarootless}
The lattice $L_{\b}$ is rootless, so is isomorphic to the Leech lattice.
\end{lem}
\pf Suppose that $M$ contains a root, say $r$. Then $r$ projects to
minimal vectors in each of $\dual Q$ and $\dual R$. It follows that
$M$ contains roots $gr$, for all $g\in K$.  Therefore, by
transitivity \refpp{ww}, $M$ contains at least $3\times 240=720$
roots. In fact, we can show that these roots form a root system of
type $\Phi (E_8^3)$. This follows from the discussion of
\refpp{mphi}.

We now quote \refpp{ww}(iii) to conclude that $u\in \pi (O(Q)\times
O(R))$.  However, this is
impossible because of \refpp{maxa9weyle8}(ii). \eop

\begin{lem}\labtt{lbetagroup}
The common stabilizer $O(L_{\a})\cap O(L_{\b})$ is isomorphic to
$Sym_3 \times 2{\cdot }Alt_9$.
\end{lem}
\pf
The intersection $O(L_{\a})\cap O(L_{\b})$ can not be $O(L_{\a})$.
Now use \refpp{maxa9weyle8}.
\eop

\appendix  \centerline{\bf \Large Appendices}

\section{Alternate proof that $2{\cdot}Alt_9$ occurs for a gluing}

In this section, we assume existence of $\L$, the Leech lattice,  and some of its properties.

We start with $Q\cong A_2\otimes E_8$ and $R:=\sqrt 3 E_8$.
There is an embedding $Q\perp R \le \Lambda$: if $h\in O(\L )$ has order 3 and trace 0, we take $R$ to be its fixed point sublattice and $Q$ to be $ann_{\L}(R)$.
Let $K$ be the common stabilizer of these three lattices.
The two projections of $\Lambda$ are $K$-maps and so are the associated maps of $\Lambda /3 \Lambda$ to $\dg Q$ and $\dg R$.  In fact, $\dg Q$ and $\dg R$ are isometric $K$-modules.
The group $K$ acts on $\dg Q$ and $\dg R$ completely reducibly, with constituents of dimensions 1 and 7.  To a gluing is associated an isometry $\varphi : \dg Q\rightarrow \dg R$

Now suppose that the $\varphi$ comes from the $E_8$-structure on $Q$ and $P$.   That is, $Q$ contains three copies of $EE_8$, say $A, B, C$, and their annihilators are isometric to $\sqrt 6 E8$.

Take $A':=ann_Q(A)$.   Let $\psi : A' \rightarrow R$ be any isomorphism of free abelian groups which is a scaled isometry (so that the scale factor is $\sqrt 2$).
Then $Q\perp R$ and $\{v +v\psi | v\in \third A' \}$ span an even unimodular lattice which has roots.  It is isometric to $E_8^3$.

To get Leech from a gluing of $Q\perp R$, we need $\varphi$ which does not arise this way.  To get an integral overlattice without roots, we need the property that if $\psi$ makes the
cosets $v+Q$ and $w+R$ correspond, then
the minimum norms $a$ in $v+Q$ and $b$ in $w+R$ must satisfy
$a+b \in 2\ZZ$ but $a+b>2$.

Since the Leech lattice exists,  it follows that there is such a $\psi$.
Since the Leech lattice is unique and since we know the isometry group of Leech,  it follows that for any $\psi$ which defines a Leech lattice, its stabilizer in $\weh$ is isomorphic to $2{\cdot}Alt_9$.

\section{Automorphism group of $V_\Lambda^+$}\label{Appendix:B}
The full automorphism group of $V_\Lambda^+$ associated with the Leech
lattice has been determined in
\cite{sh}. In this section, we recall some basic results which we used
in this article.

Let $L$ be a positive definite even lattice and
\[
1 \longrightarrow \langle\kappa\rangle \longrightarrow \hat{L} \bar{
\longrightarrow} L \longrightarrow 1
\]
a central extension of $L$ by $\langle\kappa\rangle$ such that
$\kappa^2=1$ and the commutator map $c_0(\a, \b)={\la \a, \b\ra}
\mod 2$, $\a,\b\in L$. The following theorem is well-known (cf.
\cite{FLM}):

\begin{thm}\label{thm:B1} For an even lattice $L$, the sequence
\[
1 \rightarrow \mathrm{Hom}(L, \ZZ_2)  \rightarrow
Aut(\hat{L}) \xrightarrow{\pi}
Aut(L)\rightarrow  1
\]
is exact.
In particular $Aut(\hat{\Lambda})\cong 2^{24} \cdot Co_0.$
\end{thm}

Recall that $\theta$ is the automorphism of $V_\Lambda$ defined by
\[
\theta( \a_1(-n_1)\cdots \a_k(-n_k) e^\a)= (-1)^{k}
\a_1(-n_1)\cdots \a_k(-n_k)\theta (e^\a),
\]
where $\theta(a)=a^{-1}\kappa^{\la \bar{a},\bar{a}\ra/2}$ on
$\hat{L}$.

\begin{lem}[\cite{sh}]\label{lem:4.11}
Let $L$ be a positive definite even lattice without roots, i.e.,
$L(1)=\emptyset$. Then the centralizer
$C_{Aut{V_L}}(\theta)$ of $\theta$ in $Aut{V_L}$ is isomorphic to
$Aut{(\hat{L})}$. If $L=\Lambda$ is the Leech lattice, we have
\[
C_{Aut{V_\Lambda}}(\theta)\cong Aut{(\hat{\Lambda})}\cong
2^{24}\cdot Co_0.
\]
\end{lem}

\begin{thm}[\cite{sh}]
Let $V_\Lambda^+=\{ v\in V_\Lambda|\, \theta(v)=v\}$ be the fixed
point subVOA of $\theta$ in $V_\Lambda$. Then
$
Aut{V_\Lambda^+}\cong C_{Aut{V_\Lambda}}(\theta)/ \langle\theta\rangle \cong
2^{24}\cdot Co_1
$
and the sequence
\[
1 \rightarrow  \mathrm{Hom}(\Lambda, \ZZ_2) \rightarrow
Aut{V_\Lambda^+}\ \xrightarrow{\pi} \ Aut(\Lambda)/\langle\pm
1\rangle\rightarrow  1.
\]
is exact.
\end{thm}

Next we shall recall the properties of the corresponding Miyamoto involutions.

\begin{lem}\label{L13} Let $L$ be an even lattice without roots and $e$
a \cvcch in $V_L^+$.
Then, $\tau_e\in C_{Aut(V_{L})}(\theta)$. In particular,  we may view
$\tau_e$ as an element in
$Aut{\hat{L}}(\cong C_{Aut(V_{L})}(\theta))$.
\end{lem}

\begin{proof} We view $\tau_e$ as an automorphism of $V_L$. Since $\theta$ fixes $e$, we have $\theta\tau_e\theta=\tau_{\theta(e)}=\tau_e$, which proves this lemma.
\end{proof}

\begin{rem}
 Recall the exact sequence
\[
1 \rightarrow \mathrm{Hom}(\L, \ZZ_2) \rightarrow
Aut(\hat{\L}) \xrightarrow{\pi} Aut(\L){\rightarrow}  1
\]
defined in Theorem \ref{thm:B1}. Hence, by Lemma \ref{L13},
$\pi({\tau}_e)$ is an isometry of $L$ for any \cvcch $e\in V_L^+$.
\end{rem}

\begin{nota}\label{cvcchAA1EE8}
In \cite{LSY}, all \cvcch in the VOA $V_{\L}^+$ were classified.
There are two types of \cvcch.

\textbf{$\bf AA_1$-formula:} conformal vectors supported at $AA_1$-sublattices,
i.e.,$$\omega^\pm({\a})= \frac{1}4 \a(-1)^2\cdot \mathbf{1}\pm \frac{1}4(e^\a +e^{-\a}), \quad \text{ where }\a \in \Lambda(4)=\{\a\in \L\mid
\la \a,\a\ra=4\}.$$

\textbf{$\bf EE_8$-formula:} conformal vectors supported at
$EE_8$-sublattices, i.e., $$\varphi_x(e_M),\quad \text{ where } \L\supset M\cong
EE_8, x\in
M^*.$$
\end{nota}

\begin{lem}
 Let $e$ be a \cvcch in $V_\L^+$.

(1) If $e=\omega^\pm(\a)$, then $\pi(\tau_e)=1$. In fact, $\tau_e=
\varphi_{\a}$ as an automorphism of $V_L$, i.e.,
\[
 \tau_e(u\otimes e^\b) =(-1)^{\la \a, \b\ra} u\otimes e^\b \quad \text{ for }
u\in M(1), \b\in L.
\]

(2) If $e=\varphi_x(e_M)$ for some $M\cong EE_8$ in $\L$, then $\pi(\tau_e)$
defines an isometry of $\L$ which acts as $-1$ on $M$ and $1$ on $ann_\L(M)$.
\end{lem}

Now let  $V^\natural=V_{\Lambda}^+\oplus
V_{\Lambda}^{T,+}$ be the famous Moonshine. Let $z$ be the linear map of
$V^\natural$ acting
as $1$ and $-1$ on $V_{\Lambda}^+$ and $V_{\Lambda}^{T,-}$
respectively. Then $z$ is an automorphism of $V^\natural$.

\begin{lem}\label{lem:5.12} Let $\alpha\in \Lambda(4)$.
Then $\tau_{\omega^+(\a)} \tau_{\omega^-(\a)}=z$ on $V^\natural$.
\end{lem}

Next, we shall discuss the centralizer of $\tau_e$ and $z$ in
$Aut(V^\natural)$ for any \cvcch $e$ in $V_\Lambda^+$. The following
lemma is well known \cite{LS,FLM}.

\begin{lem}\label{lem:gb1} The centralizer of $z$ in $\Aut V^\natural$ has the
structure $2^{1+24}.Co_1$.
\end{lem}

The proofs of the following two theorems can be found in \cite{LS}

\begin{thm} Let $\alpha\in\Lambda(4)$.
Set $e=\omega^\varepsilon(\alpha)$, where $\varepsilon=+$ or $-$.
Then the centralizer $C_{\Aut V^\natural}(\tau_e,z)$ has the structure
$2^{2+22}.Co_2$.
\end{thm}

If $e=\varphi_x(e_M)$, it turns out that the centralizer of $\tau_e$
in $C_{\Aut V^\natural}(z)$ also stabilizes the VOA $V_M^+\subset
V_\Lambda^+$. Moreover, we have

\begin{thm}\label{thm:E8} Let $M$ be a sublattice of $\Lambda$ isomorphic to
$EE_8$
and $x$ a vector in $M/2$.
Set $e=\varphi_x(e_M)$.
Then the centralizer $C_{\Aut V^\natural}(\tau_e,z)$ has the structure
$2^{2+8+16}.\Omega^+(8,2)$.
\end{thm}

\begin{rem}
Let $M$ be a sublattice of $\Lambda$ isomorphic to $EE_8$. Then the
stabilizer of $e_M$ in the subgroup
$Aut(\hat{M})/\langle\theta\rangle$ of $\Aut V_M^+$ is isomorphic to
$Aut(M)/\langle-1\rangle\cong O^+(8,2)$. In Theorem \ref{thm:B1},
the centralizer $C_{\Aut V^\natural}(\tau_{e_M},z)$ actually acts on
$V_M^+$ as $\Omega^+(8,2)$, which is the quotient of the commutator
subgroup of the Weyl group of $E_8$ by its center.

Let $(M, M')$ be an $EE_8$-pair in $\Lambda$. Then we have
\[
C_{Aut(V^\natural)}(\tau_{e_{M}},
\tau_{e_{M'}},z)=C_{Aut(V^\natural)}(\tau_{e_{M}}, z)\cap
C_{Aut(V^\natural)}(\tau_{e_{M'}},z).
\]
In this case,
$C_{Aut(V^\natural)}(\tau_{e_{M}},
\tau_{e_{M'}},z)$ must contain a factor
group which is isomorphic to the common stabilizer of $M$ and $M'$ in
$Aut(\L)/{\pm 1}$.
\end{rem}

\section{Niemeier lattices that contain $Q\perp R$}\label{QRinN}

In this section, we shall list the Niemeier lattices that contain
$Q\perp R$ such that $R$ is the fixed point sublattice of an
isometry of order $3$ and $Q$ is its annihilator.

%\chsays{rewrote}
Our setting is as follows: Let $h$ be an order $3$ element of
$Weyl(A_2)\cong Sym_3$. Then $h$ defines an isometry on
$Q=A_2\otimes E_8$ by $h(\a\otimes \b) =(h\a)\otimes \b$. It also
induces an isometry on $Q\perp R$ and $(Q\perp R)^*$ by acting
trivially on $R$.

Now let $N$ be a Niemeier lattice that contains $Q\perp R$. We
assume that $N$ is stable under $h$ and $Q$ and $R$ are direct
summands in $N$. In this case, the fixed point sublattice of $h$ in
$N$ is exactly $R\cong EEE_8$ and the annihilator of $R$ is
$ann_N(R)= Q = A_2\otimes E_8$.

The list of all possible Niemeier lattices (including the Leech
lattice) that satisfy the above is given below.

\[
\begin{array}{c|c}
 \hline
\text{Type of Niemeier} &  C_{O(N)}(h)\cr \hline
A_8^3  &3\times (2 \times Sym_9)  \cr
 D_8^3 & 3\times 2^7{:}Sym_8\cr
E_8^3  &3\times Weyl(E_8)\cr
A_2^{12}  & 3\times ( (Sym_3)^4. C_{Aut(\mathcal{TG})}(h)) \cr
A_1^{24} & 3\times 2^8.L_2(7)\cr
D_4^6 & 3\times (Weyl({D_4})\times Weyl({D_4})).3  \cr
\L  &3\times 2{\cdot}Alt_9  \cr
\hline
\end{array}
\]
Here $\mathcal{TG}$ is the ternary Golay code.

\medskip

\noindent\textbf{Sketch of the proof.}

Let $N=N(\Phi)$ be a Niemeier lattice associated to a root system $\Phi$.

We shall first search for the element of $O(N)$ of order $3$, which acts fixed
point freely on roots.

Let $h\in O(N)$ be such an element. Suppose $h$ preserves an irreducible component of $\Phi$, say,
$\Phi_1$.    Then $g$ also acts on the corresponding root sublattice $L_1:=span_\ZZ\{\Phi_1\}$.
In this case, $\Phi_1$ is isomorphic to $A_{3n}$, $D_{3n}$, $D_{3n+1}$, $E_6$ or $E_8$ since $h$ acts fixed point freely on roots. Then by case by case checking, $ann_{L_1}(L_1^h)$ must contain roots.

Therefore, $h$ induces a permutation on the irreducible components of $\Phi$ and has no fixed points.
Thus, $\Phi$ must be one of the followings:
\[
A_8^3,\ D_8^3,\ E_8^3, \ A_4^6,\  D_4^6,\  A_2^{12},\  A_1^{24}\  \text{ or }\ \emptyset.
\]

For $N=N(A_4^6)$, the glue code $C$ is generated by $(101441)$, $(114410)$, $(144101)$, $(141014)$, and $(110144)$ and $[N(A_4^6):A_4^6]=125$. Nevertheless, there is no element in $Sym_6$ of cycle shape $3^2$ which preserves $C$. Thus, $N(A_4^6)$ is also out.

\medskip

The explicit embedding of $Q\perp R$  for the remaining cases are given below.

\medskip

\noindent{\bf Case: $\bf  N(A_8^3)$}

$[N(A_8^3): A_8^3]=3^3$ and the glue code is generated by $(114), (141)$ and
$(411)$.

Then $Aut\,N(A_8^3)\cong W(A_8^3). (2\times S_3)$. Let $\sigma$ be the cyclic
permutation of the 3 copies of $A_8$.

Set
\[
 R=span\{ (\a,\a,\a| \a\in A_8\}\cup \{(\gamma, \gamma, \gamma)\}\cong
\sqrt{3}E_8,
\]
and
\[
 Q=span\{(\a,-\a,0), (0, \a, -\a)| \a\in A_8\}\cup \{(\gamma, -\gamma, 0),
(0,\gamma, -\gamma)\}\cong A_2\otimes E_8,
\]
where $\gamma=\frac{1}3(1^6, -2^3)$.  Note that $(3,3,3), (3,-3,0)$ and
$(0,3,-3)$ are in the glue code and $R = N(A_8^3)^\sigma$.

In this case,
\[
 C_{Aut\, N(A_8^3)} (\sigma) = 3\times (2\times Sym_9),
\]
where $Sym_9$ acts diagonally on $A_8^3$.

\medskip

\noindent{\bf Case: $\bf N(D_8^3)$}

$[N(D_8^3:D_8^3]=2^3$ and the glue code is generated by $(122), (212), (221)$.

Then $Aut\, N(D_8^3)\cong Weyl(D_8)\wr S_3$. Let $\sigma$ be the
cyclic permutation of the 3 copies of $D_8$. Then
\[
 N(D_8^3)^\sigma= R= span\{(\a,\a,\a)|\a\in D_8\}\cup \{(\gamma, \gamma,
\gamma)\}\cong \sqrt{3}E_8
\]
and
\[
\begin{split}
 Q&=ann_N(R)\\
&=span \{ (\a,-\a,0), (0,\a,-\a) | \a\in D_8\} \cup \{ (\gamma',-
\gamma,0), (0,\gamma', -\gamma'\}\\
&\cong A_2\otimes E_8,
\end{split}
\]
where $\gamma=\frac{1}2(11111111), \gamma'=\frac{1}2(1111 111-1)$.

In this case,
\[
 C_{Aut\,N(D_8^3)}(\sigma)= 3\times Weyl({D_8})\cong 3\times (2^7: Sym_8).
\]

\noindent{\bf Case: $\bf E_8^3$}

In this case, $Aut\, E_8^3\cong Weyl({E_8})\wr S_3$. Let $\sigma$ be
the cyclic permutation of the 3 copies of $E_8$. Then
\[
(E_8^3)^\sigma= R= span\{(\a,\a,\a)\mid\a\in E_8\}\cong EEE_8
\]
and
\[
\begin{split}
 Q&=ann_{E_8^3}(R)=span \{ (\a,-\a,0), (0,\a,-\a) \mid \a\in E_8\} \cong A_2\otimes E_8.
\end{split}
\]
Moreover,
\[
 C_{Aut\,E_8^3)}(\sigma)\cong 3\times Weyl({E_8}).
\]

\noindent{\bf Case: $\bf  N(A_1^{24})$}

The glue code is isomorphic to $G_{24}$ and $[N(A_1^{24}):
A_1^{24}]=2^{12}$.

In this case, $Aut\, N(A_1^{24}) \cong 2^{24}.M_{24}$.

Let $\sigma$ be the order 3 automorphism which has the shape $3^8\in M_{24}$.
Let $\mathcal{C}$ be the subcode generated by the $14$ dodecads fixed by
$\sigma$. Then $\mathcal{C}$ is isomorphic to the tripled Hamming code.
Then,
\[
 N(A_1^{24})^\sigma =R= span\{ (\a,\a,\a)| \a\in A_1^8\}\cup \{\frac{1}2 \a_C|
C\in \mathcal{C}\}\cong \sqrt{3}E_8.
\]
Let $\mathcal{T}$ be the set of sextets that are fixed by $\sigma$ and $\{O_1,
O_2, O_3\}$ the trio fixed by $\sigma$. Then $|\mathcal{T}|=7$. Set
\[
 \mathcal{H}=\{T\subset O_1| T \text{ is a tetrad of a sextet in
}\mathcal{T}\}\cup\{\emptyset, O_1\}
\]
Then $\mathcal{H}$ is isomorphic to the Hamming code.
Then
\[
\begin{split}
 Q&=span\{\a, -\a,0), (0,\a, -\a)|\a\in A_1^8\}\cup \{
\frac{1}2(\a_T,-\a_T,0),\frac{1}2 (0,\a_T,-\a_T)| T\in \mathcal{H}\}\\
&\cong
A_2\otimes E_8
\end{split}
\]

In this case, the centralizer is
\[
 C_{Aut\, N(A_1^{24})}= 2^8. C_{M_{24}}(\sigma)= 3\times 2^8. L_2(7)
\]

\noindent{\bf Case: $\bf  N(D_4^6)$}

$D_4^*/D_4\cong F_4$ and the glue code is the Hexacode.

Let $\sigma=(135)(246)$. Then $\sigma$ fixes a subcode $I$
generated by $(1\omega\ 1\omega\ 1\omega )$ and $|I|=2^2$.
Set
\[
 R=span\{ (\a,\b,\a, \b, \a, \b)| \a,\b\in D_4\}\cup \{ \a_c| c\in I\} \cong
\sqrt{3}E_8,
\]
where $\a_c=([c_1], [c_2], \dots, [c_6])$ if $c=(c_1, c_2, \dots, c_6)$ and
$[1]=(000-1), [\omega]=\frac{1}2(1111), [\bar{\omega}] = \frac{1}2(-1-1-1 1)$.
\[
 \begin{split}
  Q=&span\{ (\a, \b, -\a, -\b, 0,0), (0,0, \a, \b,-\a, -\b)| \a, \b\in D_4
\}\\
&\quad \cup \{ ([c],[c], -[c],-[c],0,0), (0,0, [c],[c],-[c],-[c])| c\in \{1,
\omega, \bar{\omega}\}\}\\
&\cong A_2\otimes E_8.
 \end{split}
\]

\[
 C_{Aut\, N(D_4^6)} = (Weyl(D_4)\times Weyl({D_4})). C_{Aut\ Hexacode}(\sigma) \cong
(2^3Sym_4\times2^3 Sym_4).3.2
\]

\medskip

\noindent{\bf Case: $\bf  N(A_2^{12})$}

The glue code is the ternary Golay code $\mathcal{TG}$ and
$[N(A_2^{12}):A_2^{12}]=3^6$. $Aut(\mathcal{TG})=2.M_{12}$.

Let $\sigma$ be an order $3$ element given by
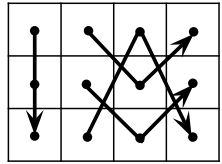
\[
% Generated with LaTeXDraw 2.0.2
% Fri May 08 11:13:03 CST 2009
% \usepackage[usenames,dvipsnames]{pstricks}
% \usepackage{epsfig}
% \usepackage{pst-grad} % For gradients
% \usepackage{pst-plot} % For axes
\scalebox{0.7} % Change this value to rescale the drawing.
{
\begin{pspicture}(0,-1.5)(4.0,1.5)
\definecolor{color0c}{rgb}{0.5019607843137255,0.5019607843137255,
0.5019607843137255}
\rput(0.0,-1.5){\psgrid[gridwidth=0.028222222,subgridwidth=0.014111111,
gridlabels=0.0pt,subgriddiv=0,subgridcolor=color0c](0,0)(0,0)(4,3)}
\psdots[dotsize=0.18](0.5,0.98)
\psdots[dotsize=0.18](0.5,-0.04)
\psdots[dotsize=0.18](0.5,-1.02)
\psdots[dotsize=0.18](1.52,0.98)
\psdots[dotsize=0.18](1.48,-0.04)
\psdots[dotsize=0.18](1.5,-1.04)
\psdots[dotsize=0.18](2.5,0.96)
\psdots[dotsize=0.18](2.5,-0.06)
\psdots[dotsize=0.18](2.5,-1.06)
\psdots[dotsize=0.18](3.5,-1.04)
\psdots[dotsize=0.18](3.5,-0.02)
\psdots[dotsize=0.18](3.52,0.96)
\psline[linewidth=0.07cm](0.5,0.96)(0.5,-0.14)
\psline[linewidth=0.07cm,arrowsize=0.2cm
2.0,arrowlength=1.4,arrowinset=0.4]{->}(0.5,-0.1)(0.5,-1.0)
\psline[linewidth=0.07cm](1.46,1.02)(2.5,-0.08)
\psline[linewidth=0.07cm](1.46,-0.06)(2.56,-1.14)
\psline[linewidth=0.07cm](1.48,-1.04)(2.5,1.02)
\psline[linewidth=0.07cm,arrowsize=0.2cm
2.0,arrowlength=1.4,arrowinset=0.4]{->}(2.5,-0.06)(3.52,0.92)
\psline[linewidth=0.07cm,arrowsize=0.2cm
2.0,arrowlength=1.4,arrowinset=0.4]{->}(2.5,-1.08)(3.52,-0.04)
\psline[linewidth=0.07cm,arrowsize=0.2cm
2.0,arrowlength=1.4,arrowinset=0.4]{->}(2.52,0.96)(3.48,-1.04)
\end{pspicture}
}
\]
Then $\sigma$ fixes a subcode $\mathcal{C}$ generated by
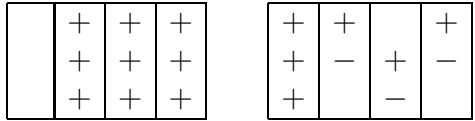
\[
 \begin{array}{|c|c|c|c|}
  \hline
  \ \ & + & + & + \\
  \ & + & + & + \\
  \ & + & + & + \\
\hline
 \end{array}
\qquad
 \begin{array}{|c|c|c|c|}
  \hline
  +& + &  &  + \\
  + & - & + & - \\
  + &  & -  & \\
\hline
 \end{array}
\]
Then  $\mathcal{C}$ is isomorphic to a triple of tetra-code and
\[
 R=span\{(\a, \a,\a)| \a\in A_2^4\} \cup \{\a_c|c\in \mathcal{C}\}\cong
\sqrt{3}E_8,
\]
where $\a_c=([c_1], [c_2], \dots, [c_{12}])$ and
$[0]=(0,0,0)$, $[1]=\frac{1}3(1,1,-2)$,\\ $[2]=\frac{1}3(-1,-1,2)$.

Let $\mathcal{H}$ be the subcode generated by
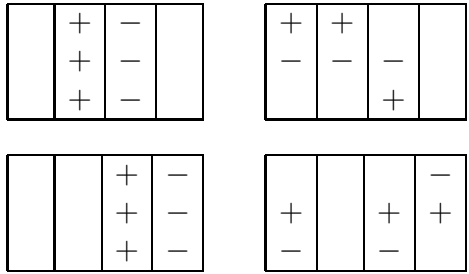
\[
 \begin{array}{|c|c|c|c|}
  \hline
  \ \ & + & - &\ \  \\
  \ & + & - &  \\
  \ & + & - &  \\
\hline
 \end{array}
\qquad
 \begin{array}{|c|c|c|c|}
  \hline
  +& + &  &  \ \ \\
  - & - & - &  \\
   &  &  +& \\
\hline
 \end{array}
\]
\[
 \begin{array}{|c|c|c|c|}
  \hline
  \ \ &\ \ & + & -   \\
  \ && + & -  \\
  \ && + & -   \\
\hline
 \end{array}
\qquad
 \begin{array}{|c|c|c|c|}
  \hline
    &\ \ &  & - \\
  + &    &+ & + \\
  - &    &- &   \\
  \hline
 \end{array}
\]
Then
\[
 Q=ann_{N(A_2^{12})}(R) = span{ A_2\otimes (A_2^4)}\cup \{ \a_c| c\in
\mathcal{H}\},
\]
where $A_2\otimes E_8$ is generated by
\[
 \begin{array}{|c|c|c|c|}
  \hline
  \a &\ \b & -\delta &\ \  \   \\
  -\a&\gamma& -\b &   \\
  \ & \delta & -\gamma &    \\
\hline
 \end{array}
\qquad
 \begin{array}{|c|c|c|c|}
  \hline
  &\ \ \ & \delta& -\b\\
  \a & &\ \b & -\gamma   \\
  -\a&&\ \gamma& -\delta    \\
  \hline
 \end{array}
\]

In this case, the centralizer is
\[
 C_{Aut\,N(A_2^{12})}(\sigma)= (S_3\times S_3\times S_3\times S_3).
C_{Aut(\mathcal{TG})}(\sigma).
\]

\noindent{\bf Case: Leech lattice $\bf \L$}

This case was treated in \cite{glee8}.

\section{Centralizers of pairs of $2A$-involutions for the $3C$-case}
\label{App:D}

For background in this section, see \cite{grfg,gr12}

We take a $3C$-pair of $2A$ involutions, $x, y$, and study $C(x,y)$
(meaning $C_{\MM}(\la x, y \ra )$) and $C(x, y, z)$, where $z\in 2B$
and $z\in C(x, y)$.

Consider the $3C$-element $h:=xy$. We have $C(h)=F\times \la h \ra$,
where $F\cong F_3$, a simple group of order $2^{15}{\cdot}3^{10}5^3{\cdot}
7^2{\cdot}13{\cdot}19{\cdot}31$. The group $F$ has one class of involutions and
they are contained in the $2B$ class of $\MM$. We take $z\in F$ and
$\la x, y \ra = C(F)$.

Let us go to $C:=C(z)\cong 2^{1+24}{\cdot}Co_1$. This is a twisted
holomorph in the sense of \cite{grfg,grmont}. The element $h$ is in
$C$ and corresponds in $O(\Lambda )$ to an element $h'$ of order 3
which is a permutation in the natural $M_{24}$ of cycle shape $3^8$.
Its centralizer in $O(\Lambda)$ has the form $3 \times 2{\cdot
}Alt_9$. Therefore, $C_C(h)$ has shape $2^{1+8}.Alt_9$.

There exist involutions $x' , y' \in O(\Lambda )$ of trace 8 so that
$h'=x'y'$.  If we choose $x, y\in C$ to correspond to such
involutions, then $\la x, y \ra$ centralizes $C_C(h)$. Then we get
$C(x,y,z)=F$ and $\la x, y, C(h) \ra = \la x, y \ra \times F$.

\end{document}